\documentclass[12pt]{amsart}
\usepackage{amssymb}
\usepackage{tikz}
\usetikzlibrary{arrows,shapes,automata,snakes,backgrounds,petri,positioning}

\newtheorem{thm}{Theorem}

\newtheorem{cor}{Corollary}
\newtheorem{lem}{Lemma}
\newtheorem{obs}{Observation}
\newtheorem{prp}{Proposition}

\newtheorem{defn}{Definition}

\DeclareMathOperator{\rk}{rk}

\DeclareMathOperator{\R}{\mathbb{R}}
\DeclareMathOperator{\B}{B}

\title[The shard intersection order]{On the shard intersection order of a Coxeter group}

\author[T. K. Petersen]{T. Kyle Petersen}
\thanks{Work supported by an NSA Young Investigator grant}

\date{\today}

\begin{document}

\begin{abstract}
Introduced by Reading, the \emph{shard intersection order} of a finite Coxeter group $W$ is a lattice structure on the elements of $W$ that contains the poset of noncrossing partitions $NC(W)$ as a sublattice. Building on work of Bancroft in the case of the symmetric group, we provide combinatorial models for shard intersections of all classical types, and use this understanding to prove the shard intersection order is EL-shellable. 

Further, inspired by work of Simion and Ullman on the lattice of noncrossing partitions, we show that the shard intersection order on the symmetric group admits a \emph{symmetric boolean decomposition}, i.e., a partition into disjoint boolean algebras whose middle ranks coincide with the middle rank of the poset. Our decomposition also yields a new symmetric boolean decomposition of the noncrossing partition lattice.
\end{abstract}

\maketitle

\section{Introduction}

Introduced by Reading \cite{R}, \emph{shards} are certain pieces of a simplicial hyperplane arrangement $\mathcal{A}$. As will be explained in Section \ref{sec:poset}, we form shards by splitting up the hyperplane arrangement into closed, full-dimensional cones, each entirely contained in some hyperplane. These cones overlap in various ways and, remarkably, the set of intersections of the shards is in bijection with the set $\mathcal{R}$ of \emph{regions} formed by the complement of $\mathcal{A}$ in the ambient vector space. The lattice of intersections of shards, with partial order given by reverse containment, thus passes to a lattice structure $(\mathcal{R},\leq)$ on regions. 

In \cite{R}, Reading proves many general properties of this lattice, including the fact that it is graded, atomic, and coatomic. He also gives a characterization of lower intervals, computes the M\"obius number, and shows that the faces of the order complex of $(\mathcal{R},\leq)$ are in bijection with the faces of the ``pulling triangulation" of a zonotope dual to $\mathcal{A}$.

In the case where $\mathcal{A}$ is the Coxeter arrangement of a root system with Coxeter group $W$, the regions correspond to elements $w$ in $W$, and so the shard intersection order gives a new lattice structure, $(W,\leq)$, to the Coxeter group. The shards themselves are in bijection with elements of $W$ having exactly one descent and the rank generating function is given by the $W$-Eulerian polynomial, \[ W(t) = \sum_{w \in W} t^{d(w)},\] where $d(w)$ denotes the number of descents of $w$, i.e., the number of simple generators $s$ such that $\ell(ws)< \ell(w)$.  Further, the $W$-noncrossing partition lattice, $NC(W)$, is isomorphic to a sublattice of the shard intersection lattice. Reading \cite{R} first proves, then exploits this fact to deduce many properties of $NC(W)$, not least of which is the nontrivial fact that $NC(W)$ is a lattice for any $W$. The classical (type $A_{n-1}$) lattice of noncrossing partitions has been an object of great interest in combinatorics and, more recently, in algebra and geometry. See, for example, \cite{An,A,ABW,AR,IT,McC,Rei} for more on noncrossing partitions and their generalizations.

Bancroft \cite{B} has studied the lattice of shard intersections for root systems of type $A_{n-1}$, i.e., when the associated Coxeter group is the symmetric group. Her work gives an explicit combinatorial description to shard intersections in terms of so-called ``permutation pre-orders." Bancroft uses this model to show $(S_n,\leq)$ is EL-shellable and the first goal of this paper is to extend Bancroft's result to types $B_n$ and $D_n$.

\begin{thm}\label{thm:EL}
If $W$ is a finite Coxeter group of classical type ($A_{n-1}, B_n$, or $D_n$), the shard intersection order $(W,\leq)$ is EL-shellable.
\end{thm}

The EL-shellability of a poset has well-known topological implications. (See \cite[Section 3.2]{Wachs}.) For example, Theorem \ref{thm:EL} tells us the Stanley-Reisner ring of the shard intersection order is Cohen-Macaulay and, for $W$ of rank $n$, the order complex of $(W,\leq)$ has the homotopy type of a wedge of $(n-2)$-spheres. The number of spheres in the wedge product equals the absolute value of the M\"obius number of the poset, $\mu(1,w_0)$. Generally, an EL-labeling allows one to compute the M\"obius number $\mu(u,v)$ as (up to sign) the number of ``falling" maximal chains in the interval $[u,v]$. In principle, then, our EL-labelings allow for computation of M\"obius numbers, though apart from the interval $[1,w_0]$ we know of no simple characterization of falling chains. Reading computed $\mu(1,w_0)$, by different means, in \cite[Theorem 1.3]{R}. For type $A_{n-1}$ this is the number of indecomposable permutations, found in \cite[A003319]{oeis}. The related sequence for type $B_n$ is \cite[A109253]{oeis}; for type $D_n$ it is \cite[A112225]{oeis}. 

Our proof of Theorem \ref{thm:EL} is done case-by-case in Section \ref{sec:shell}. It remains to show whether the shard intersection orders for non-classical $W$ are also shellable. (Since EL-shellability respects cartesian products of posets, to prove EL-shellability for all finite Coxeter groups it suffices to extend Theorem \ref{thm:EL} to cover all irreducible Coxeter groups.) We remark that it is known that the noncrossing partition lattices $NC(W)$ are EL-shellable for all finite $W$; see \cite{ABW} for a uniform proof. Ideally, one would like to do for the shard intersection order $(W,\leq)$ what Athanasiadis, Brady, and Watt \cite{ABW} did for $NC(W)$ and provide a uniform proof of (EL-) shellability.

In Section \ref{sec:poset}, we will derive Bancroft's model for shard intersections of type $A_{n-1}$ (though our visual representation is new), and we will construct similar combinatorial models for $(W,\leq)$ when $W=B_n$ and when $W=D_n$. Thorough understanding of these models is key to our proofs of EL-shellability.

After proving Theorem \ref{thm:EL} we turn our attention to the shard intersection order of the symmetric group $S_n$ (type $A_{n-1}$) and show that it admits what Hersh \cite[Section 3]{H} calls a \emph{symmetric boolean decomposition}, i.e., a partition of the poset into disjoint boolean algebras whose middle ranks coincide with the middle rank of the lattice. This idea first appears in work of Simion in Ullman \cite{SU}, who demonstrate such a decomposition for the noncrossing partition lattice $NC(n)\cong NC(A_{n-1})$ as part of their proof that $NC(n)$ has a symmetric chain decomposition. Hersh \cite[Theorem 7]{H} proved that $NC(B_n)$ has a symmetric boolean decomposition as well. See also \cite{BP}. The second main result of this paper is the following.

\begin{thm}\label{thm:SBD}
The shard intersection order $(S_n,\leq)$ admits a symmetric boolean decomposition. Moreover, this decomposition restricts to a symmetric boolean decomposition for $(NC(n),\leq)$.
\end{thm} 

We remark that the decomposition of $NC(n)$ we obtain is different from Simion and Ullman's decomposition.

We provide the precise definition of symmetric boolean decomposition of a poset in Section \ref{sec:sbd}, and outline some of its consequences. For example, if a poset admits a symmetric boolean decomposition, its rank generating function is what is known as \emph{$\gamma$-nonnegative}.
 
Section \ref{sec:proof} gives the proof of Theorem \ref{thm:SBD}, which follows from Foata and Strehl's ``valley-hopping" action on permutations. We conclude in Section \ref{sec:generalization} with remarks about extending known results on symmetric boolean decompositions, both for the shard intersection orders and noncrossing partition lattices.  While Theorem \ref{thm:SBD} applies only to the symmetric group, our hope is that the models for $B_n$ and $D_n$ given in Section \ref{sec:poset} might ultimately help to extend it to other types.

\section{Combinatorial models for shard intersections}\label{sec:poset}

We begin with a description of shards in a Coxeter arrangement. We assume the reader is familiar with standard Coxeter group definitions. See \cite{Hum} for more background. 

Let $W$ be a finite Coxeter group with root system $\Phi = \Phi^+ \cup \Phi^-$ and simple roots $\Delta = \{\alpha_1,\ldots, \alpha_n\}$. Suppose, without loss of generality, that $W$ acts on a vector space $V$ with basis $\Delta$ and inner product $\langle \cdot,\cdot\rangle$. 

Given a root $\beta \in \Phi$, the hyperplane $H_{\beta}$ is the codimension-1 subspace given by \[ H_{\beta} = \{ \lambda \in V: \langle \lambda, \beta \rangle = 0 \}. \] Notice, then, that $H_{\beta} = H_{-\beta}$ and it suffices to consider only positive roots $\beta \in \Phi^+$. The group $W$ is the group generated by orthogonal reflections across these hyperplanes.

The \emph{Coxeter arrangement} is the collection of all such hyperplanes: \[ \mathcal{A} = \mathcal{A}(\Phi) = \{ H_{\beta} : \beta \in \Phi^+\}.\] The complement of $\mathcal{A}$ in $V$ is a collection of open cones: $\mathcal{R}=V\setminus \bigcup_{H \in \mathcal{A}} H $ that we call the \emph{regions} (or chambers) of the arrangement.

The choice of simple roots implicitly defines a \emph{base region} $B$ (also known as the fundamental chamber) given by \[ B = \{ \lambda \in V : \langle \lambda, \beta \rangle > 0 \mbox{ for all } \beta \in \Phi^{+}\}.\] We obtain a natural bijection between regions $R$ in $\mathcal{R}$ and elements $w$ of $W$ by $R \leftrightarrow w$ if and only if $wB=R$.

We now define shards. While we follow \cite{R}, our presentation is specific to the case of finite Coxeter groups and their root systems.

In order to define a shard of $\mathcal{A}$, we need to consider all rank two subarrangements generated by two hyperplanes, say $H$ and $H'$.  We denote such a subarrangement by $\mathcal{A}(H,H')$. As discussed in \cite[Section 3]{R}, there is precisely one region $B'$ in $\mathcal{A}(H,H')$ that contains the base region $B$ of $\mathcal{A}$. There are precisely two hyperplanes on the closure of $B'$ in $\mathcal{A}(H,H')$, and we call these hyperplanes \emph{basic} with respect to $\mathcal{A}(H,H')$. See Figure \ref{fig:cutting}.

\begin{figure}
\begin{tikzpicture}
\draw[very thick] (-2.5,0) -- (2.5,0) node[right] {\tiny $H''\cap \{ \langle \lambda, \beta \rangle \geq 0 \geq \langle \lambda, \beta' \rangle \}$};
\draw (-2.5,0) node[left] {\tiny $H''\cap \{ \langle \lambda, \beta \rangle \leq 0 \leq \langle \lambda, \beta' \rangle \}$};
\draw[very thick] (-2,2) -- (2,-2);
\draw[color=white,fill=white] (0,0) circle (5pt);
\draw[very thick] (0,-2.5) -- (0,3.25) node[above] {$H_{\beta}$};
\draw[very thick] (-2,-2) -- (3,3) node[right] {$H_{\beta'}$};
\draw (.75,1.5) node {$B'$};
\draw[->] (2.5,2.5) -- (2.8,2.2) node[right] {\tiny $\langle \lambda,\beta' \rangle < 0$};
\draw[->] (2.5,2.5) -- (2.2,2.8) node[above] {\tiny $\langle \lambda,\beta' \rangle > 0$};
\draw[->] (0,2.5) -- (-.35,2.5) node[left] {\tiny $\langle \lambda,\beta \rangle < 0$};
\draw[->] (0,2.5) -- (.35,2.5) node[right] {\tiny $\langle \lambda,\beta \rangle > 0$};
\end{tikzpicture}
\caption{A rank two subarrangement $\mathcal{A}(H,H')$ with basic hyperplanes $H_{\beta}$ and $H_{\beta'}$ determined by the fact that $B \subseteq B'$. We see the hyperplane $H''$ being cut by both $H_{\beta}$ and $H_{\beta'}$.}\label{fig:cutting}
\end{figure}

We will let $H_{\beta}$ and $H_{\beta'}$ denote the basic hyperplanes of $\mathcal{A}(H,H')$, with $\beta, \beta' \in \Phi^+$. Since the base region $B$ is contained in $B'$ we know that both $\beta$ and $\beta'$ point ``toward" $B'$ as shown in Figure \ref{fig:cutting}. More importantly, we will now ``cut" each non-basic hyperplane $H''$ in $\mathcal{A}(H,H') = \mathcal{A}(H_{\beta},H_{\beta'})$ according to whether the inner product of a generic point $\lambda \in H''$ is positive or negative with respect to $\beta$ and $\beta'$. Again, see Figure \ref{fig:cutting}.

Thus, for each hyperplane $H$ in $\mathcal{A}$, there are a number of rank two arrangements in which $H$ appears. In some of these arrangements $H$ is basic, while in others $H$ is not basic and gets cut. Taken together, the cuts divide $H$ into a collection of open, full-dimensional cones. The closures of these open cones are the shards formed from $H$.

If we cut by the basic hyperplane $H_{\beta}$, we split $H$ by imposing either the condition $\langle \cdot, \beta \rangle \geq 0$ or the condition $\langle \cdot, \beta \rangle \leq 0$. Thus a shard $\Sigma \subseteq H$ is given by:
\[ \Sigma = H \bigcap_{\tiny \substack{  H_{\beta} \mbox{ cuts } H \\ \beta \in \Phi^+ }} \{ \lambda \in V : \langle \lambda, \epsilon\beta \rangle \geq 0 \},\] where $\epsilon$ is either $+$ or $-$ for the chosen $\beta$. While any shard corresponds to a unique choice of signs $\epsilon$, not every choice of sign corresponds to a shard. In particular, the number of shards contained in $H$ need not be a power of two.

We now proceed to specific root systems.

\subsection{Type $A_{n-1}$}

While Bancroft \cite{B} has already produced a combinatorial model for shard intersections in type $A_{n-1}$, we include the details here so that the type $B_n$ and $D_n$ models emerge as a natural outgrowth of the ideas behind the type $A_{n-1}$ model.

The root system of type $A_{n-1}$ is most naturally realized in 
\[V =  \R^n/\langle (1,1,\ldots,1)\rangle \cong \left\{ (x_1,\ldots,x_n) \in \R^n :  \sum_{i=1}^{n} x_i = 0 \right\} \cong \R^{n-1},\] where we let \[ \Phi^+ = \{ \varepsilon_j - \varepsilon_i : 1 \leq i < j \leq n\},\] and where $\varepsilon_1,\ldots, \varepsilon_n$ is the standard orthonormal basis of $\R^n$.

With respect to this choice of root system, the base region $B$ is given by: \[ B = \{ (x_1,\ldots, x_n) \in \R^n : x_1 < x_2 < \cdots < x_n\},\] and the hyperplane corresponding to a positive root $\varepsilon_j - \varepsilon_i$ is: \[ H_{i,j} = \{ (x_1,\ldots,x_n) \in \R^n : x_i = x_j \}.\] 

There are only two types of rank two subarrangements of $\mathcal{A}(A_{n-1})$, shown in Figure \ref{fig:Asubs}. If we consider two hyperplanes $H_{i,j}$ and $H_{k,l}$ with $\{i,j\} \cap \{k,l\} = \emptyset$, the hyperplanes are orthogonal, and we get no cutting relations. (This arrangement is isomorphic to $\mathcal{A}(A_1\times A_1)$.) On the other hand, suppose the hyperplanes are not orthogonal. Then we get an arrangement isomorphic to $\mathcal{A}(A_2)$, generated, say, by $H_{i,j}$ and $H_{j,k}$, with $1\leq i < j<k \leq n$. The hyperplanes $H_{i,j}$ and $H_{j,k}$ are basic (as shown in Figure \ref{fig:Asubs}), and the third hyperplane in the arrangement, $H_{i,k}$, gets cut according to whether $x_j \leq x_i = x_k$ or $x_i = x_k \leq x_j$.

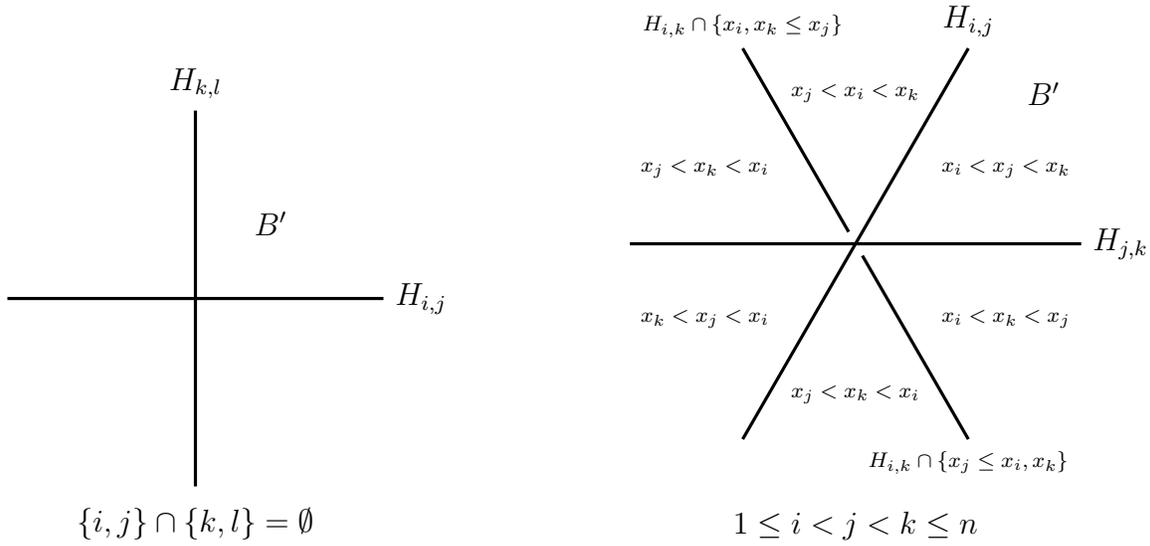
\begin{figure}
\begin{tikzpicture}
\draw[very thick] (-2.5,0) -- (2.5,0) node[right] {$H_{i,j}$};
\draw[very thick] (0,-2.5) -- (0,2.5) node[above] {$H_{k,l}$};
\draw (1,1) node {$B'$};
\draw (0,-3) node {$\{i,j\}\cap \{k,l\} = \emptyset$};
\end{tikzpicture}
\hspace{2cm}
\begin{tikzpicture}
\draw[very thick] (-1.5,2.6) -- (1.5,-2.6) node[below] {\tiny $H_{i,k}\cap \{ x_j \leq x_i, x_k\}$};
\draw (-1.5,2.6) node[above] {\tiny $H_{i,k}\cap \{ x_i, x_k \leq x_j\}$};
\draw[fill=white, color=white] (0,0) circle (5pt);
\draw[very thick] (-3,0) -- (3,0) node[right] {$H_{j,k}$};
\draw[very thick] (-1.5,-2.6) -- (1.5,2.6) node[above] {$H_{i,j}$};
\draw (2.5,2) node {$B'$};
\draw (2,1) node {\tiny $x_i < x_j < x_k$};
\draw (2,-1) node {\tiny $x_i < x_k < x_j$};
\draw (0,2) node {\tiny $x_j < x_i < x_k$};
\draw (0,-2) node {\tiny $x_j < x_k < x_i$};
\draw (-2,1) node {\tiny $x_j < x_k < x_i$};
\draw (-2,-1) node {\tiny $x_k < x_j < x_i$};
\draw (0,-3.75) node {$1\leq i < j < k \leq n$};
\end{tikzpicture}
\caption{The rank two subarrangements of $\mathcal{A}(A_{n-1})$.}\label{fig:Asubs}
\end{figure}

As there are no other possibilities for the rank two subarrangements of $\mathcal{A}(A_{n-1})$, we have the following Proposition, also stated in \cite{B}.

\begin{prp}\label{prp:Ashards}
The hyperplane $H_{i,j}$ has $2^{j-i-1}$ shards, and hence there are \[ \sum_{1\leq i < j \leq n} 2^{j-i-1} = 2^n - n-1\] shards in $\mathcal{A}(A_{n-1})$.
\end{prp}

Now that we have identified the shards, we will describe how permutations encode shard intersections.

Recall that the Coxeter group $A_{n-1}$ is isomorphic to the group $S_n$ of permutations of $\{1,\ldots, n\}$. We write permutations in one line notation, i.e., $w =w(1)\cdots w(n) \in S_n$. Recall that a \emph{descent} of a permutation is a letter $w(i)$ such that $w(i) > w(i+1)$ and an \emph{ascent} is a letter $w(i)$ such that $w(i) < w(i+1)$. Let $d(w)$ denote the number of descents of $w$. A (maximal) \emph{decreasing run} of a permutation is a word found between consecutive ascents. We can highlight the decreasing runs by inserting vertical bars in ascent positions. For example, if $w = 283964517$, we write $w = 2|83|964|51|7$, and we have $d(w) = 4$.

Visually, we represent the permutation as an array with a mark in column $i$ (from left to right), row $j$ (from bottom to top) if $w(i) = j$. We group together any decreasing runs into blocks with thick lines:
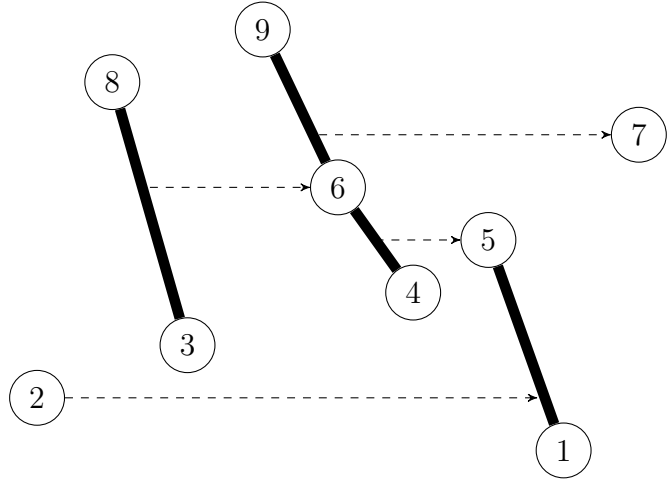
\begin{equation}\label{eq:array}
\begin{tikzpicture}[>=stealth',bend angle=45,auto,xscale=1, yscale=.7]
\tikzstyle{state}=[circle,draw=black,minimum size=4mm]
\draw (1,8) node[state] (8) {$8$};
\draw (2,3) node[state] (3) {$3$};
\draw (3,9) node[state] (9) {$9$};
\draw (4,6) node[state] (6) {$6$};
\draw (5,4) node[state] (4) {$4$};
\draw (6,5) node[state] (5) {$5$};
\draw (7,1) node[state] (1) {$1$};
\draw (0,2) node[state] (2) {$2$};
\draw (8,7) node[state] (7) {$7$};
\draw[line width=4] (8)--(3);
\draw[line width=4] (9)--(6);
\draw[line width=4] (6)--(4);
\draw[line width=4] (5)--(1);
\draw[dashed, ->] (1.5,6)--(6);
\draw[dashed, ->] (4.5,5)--(5);
\draw[dashed, ->] (3.75,7)--(7);
\draw[dashed, ->] (2)--(6.65,2);
\end{tikzpicture}
\end{equation}
If it is possible to draw a horizontal line to connect two decreasing runs, the block on the left is considered less than the block on the right. This gives a pre-order on $\{1,\ldots,n\}$ that Bancroft calls a \emph{permutation pre-order}.

In all that follows we will pass freely from thinking of $w\in S_n$ as a word and as a permutation pre-order.

As Bancroft shows, permutation pre-orders neatly encode type $A_{n-1}$ shard intersections \cite{B} as follows. For any $w \in S_n$, define the \emph{cone} of $w$, $C(w)$, as the set of points $(x_1,\ldots,x_n) \in V$ such that:
\begin{itemize}
\item if $i$ and $j$ are in the same block in $w$, then $x_i = x_j$,
\item if $i < k < j$ and $k$ is not in the same block as $i$ and $j$, then:
\begin{itemize}
\item[a)]  $x_k \leq x_i = x_j$ if $k$ appears to the left of $i$ in $w$, and
\item[b)] $x_i = x_j \leq x_k$ if $k$ appears to the right of $i$ in $w$.
\end{itemize}
\end{itemize}
The example shown in \eqref{eq:array} then corresponds to the cone of all points satisfying \[ x_8 = x_3 \geq x_9 = x_6 = x_4 \geq x_7, \]
\[ x_9 = x_6 = x_4 \geq x_5=x_1 \leq x_2, \quad \mbox{ and } \]
\[ x_1 + x_2+ x_3+x_4+x_5+x_6+x_7 + x_8+x_9 = 0.\]
At the extremes we have $C(1|2|\cdots|n) = V$ and $C(n\cdots 21)= (0,0,\ldots,0)$. Notice that the dimension of $C(w)$ is equal to one less than the number of decreasing runs in $w$ (since the sum of the coordinates is zero), and hence codimension corresponds to descent number.

\begin{obs}\label{obs:des}
For any $w \in S_n$, \[ d(w) = n-1-\dim(C(w)).\] In particular, shards correspond to elements with one descent.
\end{obs}

Thus Proposition \ref{prp:Ashards} gives a roundabout way to show there are $2^n-n-1$ permutations with one descent.

For completeness we prove the cones so described are in fact intersections of shards.

\begin{prp}\label{prp:Awelldef}(Permutation pre-orders correspond to shard intersections)
Every intersection of $A_{n-1}$ shards equals $C(w)$ for some $w \in S_n$, and for any $w \in S_n$, the cone corresponding to $w$ is an intersection of shards.
\end{prp}

\begin{proof}
The identity permutation corresponds to the empty intersection, i.e., $V \cong \R^{n-1}$, while if $w$ has only one descent, it corresponds to a shard itself and we are done.

Now we will show that for any collection of shards $\{\Sigma_1,\ldots,\Sigma_r\}$, there is an element $w$ such that \[ \bigcap_{i=1}^r \Sigma_i = C(w).\] For induction, suppose the result holds for any intersection of fewer shards. In particular, $\bigcap_{i = 1}^{r-1} \Sigma_i = C(u)$ for some $u$. Let $\Sigma_r = C(v)$ be a new shard with $x_a = x_b$. Then $\Sigma_r \cap C(u) = C(w)$, where $w$ is the permutation formed by merging the blocks of $u$ containing $a$ and $b$, along with any blocks between them. Moreover, if $a < k < b$ and $k$ was left of $a$ in $u$ but right of $a$ in $v$, then $k$ is in the same block with $a$ and $b$ in $w$.

For example, taking the pre-order in \eqref{eq:array} with the shard $31|2|4|5|6|7|8|9$ we get:
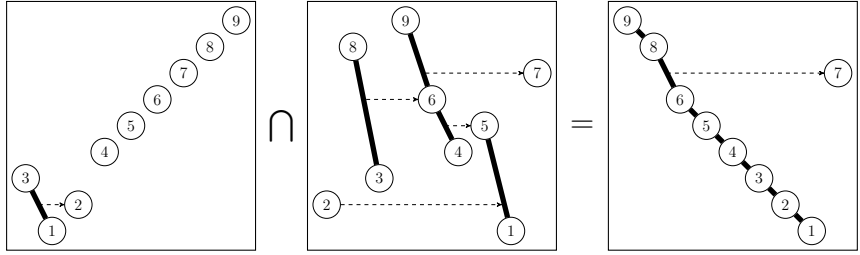
\[ \begin{tikzpicture}
\tikzstyle{state1}=[rectangle,draw=black,scale=.5]
\draw(0,0) node[state1]
{\begin{tikzpicture}[>=stealth',bend angle=45,auto,scale=.7]
\tikzstyle{state}=[circle,draw=black,minimum size=4mm]
\draw (8,8) node[state] (8) {$8$};
\draw (1,3) node[state] (3) {$3$};
\draw (9,9) node[state] (9) {$9$};
\draw (6,6) node[state] (6) {$6$};
\draw (4,4) node[state] (4) {$4$};
\draw (5,5) node[state] (5) {$5$};
\draw (2,1) node[state] (1) {$1$};
\draw (3,2) node[state] (2) {$2$};
\draw (7,7) node[state] (7) {$7$};
\draw[line width=4] (3)--(1);
\draw[dashed, ->] (1.5,2)--(2);
\end{tikzpicture}};
\draw(2,0) node {$\bigcap$};
\draw(4,0) node[state1]
{\begin{tikzpicture}[>=stealth',bend angle=45,auto,scale=.7]
\tikzstyle{state}=[circle,draw=black,minimum size=4mm]
\draw (1,8) node[state] (8) {$8$};
\draw (2,3) node[state] (3) {$3$};
\draw (3,9) node[state] (9) {$9$};
\draw (4,6) node[state] (6) {$6$};
\draw (5,4) node[state] (4) {$4$};
\draw (6,5) node[state] (5) {$5$};
\draw (7,1) node[state] (1) {$1$};
\draw (0,2) node[state] (2) {$2$};
\draw (8,7) node[state] (7) {$7$};
\draw[line width=4] (8)--(3);
\draw[line width=4] (9)--(6);
\draw[line width=4] (6)--(4);
\draw[line width=4] (5)--(1);
\draw[dashed, ->] (1.5,6)--(6);
\draw[dashed, ->] (4.5,5)--(5);
\draw[dashed, ->] (3.75,7)--(7);
\draw[dashed, ->] (2)--(6.65,2);
\end{tikzpicture}};
\draw (6,0) node {$=$};
\draw(8,0) node[state1]
{\begin{tikzpicture}[>=stealth',bend angle=45,auto,scale=.7]
\tikzstyle{state}=[circle,draw=black,minimum size=4mm]
\draw (2,8) node[state] (8) {$8$};
\draw (6,3) node[state] (3) {$3$};
\draw (1,9) node[state] (9) {$9$};
\draw (3,6) node[state] (6) {$6$};
\draw (5,4) node[state] (4) {$4$};
\draw (4,5) node[state] (5) {$5$};
\draw (8,1) node[state] (1) {$1$};
\draw (7,2) node[state] (2) {$2$};
\draw (9,7) node[state] (7) {$7$};
\draw[line width=4] (9)--(8)--(6)--(5)--(4)--(3)--(2)--(1);
\draw[dashed, ->] (2.5,7)--(7);
\end{tikzpicture}};
\end{tikzpicture}
 \]

We have shown that an intersection of shards corresponds to $C(w)$ for some $w$. Now we will show that each cone $C(w)$ corresponds to an intersection of shards. 

Permutations with one descent correspond to shards themselves, so suppose $w$ has more than one descent. The following describes a collection of shards whose intersection gives the cone $C(w)$. Given two elements in a decreasing run, say $w(i) > w(j)$ (and so $i < j$), we let $\Sigma$ be the shard with $x_{w(i)} = x_{w(j)}$ and such that for each $k$ with $w(j) < k < w(i)$, we put $x_k \leq x_{w(i)}$ if $w^{-1}(k) < i$, $x_{w(i)} \leq x_k$ otherwise.

Doing this for all pairs of elements in decreasing runs yields a collection of shards $\Sigma$, each of which contains $C(w)$ and such that all the conditions imposed by $w$ are articulated by some shard.

To illustrate, let $w = 2|83|964|51|7$. Then the collection of shards we get is:
\begin{align*}
C(2|83|964|51|7) &= C(1|2|83|4|5|6|7|9) \cap C(1|2|3|4|5|8|96|7) \\
&\quad \cap C(1|2|3|8|94|5|6|7) \cap C(1|2|3|64|5|7|8|9) \cap C(2|3|4|51|6|7|8|9).
\end{align*}

This completes the proof of the Proposition.
\end{proof}

We remark that while the idea in the proof above shows $C(w)$ is formed as an intersection of a set of shards, the set of shards we generate is neither necessarily minimal nor maximal. In our example, intersecting with the shard $1|2|3|6|8|94|5|7$ would not change the intersection. Also, we could have removed the shard $1|2|3|8|94|5|6|7$ and still obtained $C(w)$.

We can now give a partial order on $S_n$ by reverse inclusion of the corresponding subsets of $V$.

\begin{defn}[Shard intersection order on permutations]\label{def:order}
Let $u, v\in S_n$. Then $u \leq v$ in the shard intersection order if and only if $C(v) \subseteq C(u)$. 
\end{defn}

This definition can be stated combinatorially in terms of permutation pre-orders as follows. Verification of equivalence is straightforward.

\begin{prp}\label{prp:order}
In terms of permutation pre-orders, $u \leq v$ if and only if:
\begin{itemize}
\item (Refinement) $u$ refines $v$ as a set partition, and
\item (Consistency) if $i$ and $j$ are in the same block in $u$, and $i < k < j$ (with $k$ not in the same block as $i$ and $j$ in $u$), then either $k$ is in the same block as $i$ and $j$ in $v$, or $k$ is on the same side of $i$ and $j$ in $v$ as in $u$. 
\end{itemize}
\end{prp}

Intuitively, the pre-order of $v$ may be obtained by merging some blocks in the pre-order of $u$ while maintaining consistent relations. For example, although we can merge two blocks of $1|2|5|73|4|6|8$ to obtain $1|2|54|73|6|8$, they are not comparable in the partial order since in one case $4$ is right of $7$ while in the other it is to the left. On the other hand,
$1|2|5|73|4|6|8 < 2|5|73|4|861$, i.e.,  
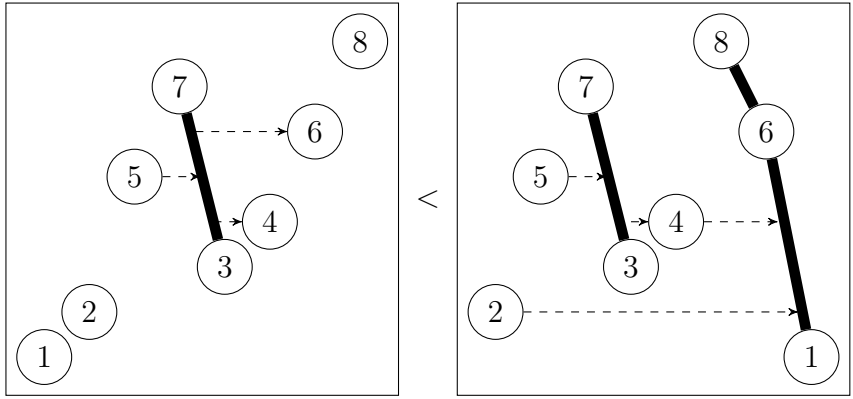
\[
\begin{tikzpicture}
\tikzstyle{state1}=[rectangle,draw=black]
\draw (0,0) node[state1]
 {\begin{tikzpicture}[>=stealth',bend angle=45,auto,scale=.6]
 \tikzstyle{state}=[circle,draw=black,minimum size=4mm]
 \draw (1,1) node[state] (1) {$1$};
 \draw (2,2) node[state] (2) {$2$};
 \draw (3,5) node[state] (3) {$5$};
 \draw (4,7) node[state] (4) {$7$};
 \draw (5,3) node[state] (5) {$3$};
 \draw (6,4) node[state] (6) {$4$};
 \draw (7,6) node[state] (7) {$6$};
 \draw (8,8) node[state] (8) {$8$};
 \draw[line width=4] (4)--(5);
 \draw[dashed, ->] (3)--(4.45,5);
 \draw[dashed, ->] (4.75,4)--(6);
 \draw[dashed, ->] (4.35,6)--(7);
 \end{tikzpicture}};
\draw (3,0) node {$<$};
\draw (6,0) node[state1]
 {\begin{tikzpicture}[>=stealth',bend angle=45,auto,scale=.6]
 \tikzstyle{state}=[circle,draw=black,minimum size=4mm]
 \draw (1,2) node[state] (1) {$2$};
 \draw (2,5) node[state] (2) {$5$};
 \draw (3,7) node[state] (3) {$7$};
 \draw (4,3) node[state] (4) {$3$};
 \draw (5,4) node[state] (5) {$4$};
 \draw (6,8) node[state] (6) {$8$};
 \draw (7,6) node[state] (7) {$6$};
 \draw (8,1) node[state] (8) {$1$};
 \draw[line width=4] (3)--(4);
 \draw[line width=4] (6)--(7);
 \draw[line width=4] (7)--(8);
 \draw[dashed, ->] (1)--(7.7,2);
 \draw[dashed, ->] (2)--(3.45,5);
 \draw[dashed, ->] (3.65,4)--(5);
 \draw[dashed, ->] (5)--(7.25,4);
 \end{tikzpicture}}; 
\end{tikzpicture}
\]
because we can obtain the pre-order on the right by merging the $8$, the $6$, and the $1$. This new block had to be to the right of the block with the $3$ and the $7$ because the $6$ was already to the right.  The new block has to be comparable to the $2$ and comparable to the $4$, but it could have appeared to the right or the left of the $2$, and to the right or the left of the $4$. These choices give rise to other permutation pre-orders (with the same set of blocks) that lie above $1|2|5|73|4|6|8$, namely, $5|73|4|861|2$, $2|5|73|861|4$, and $5|73|861|2|4$.

The intersection lattice is ranked by codimension, so by Observation \ref{obs:des} we have the following.

\begin{obs}\label{obs:rankA}
The rank of a permutation $w$ in the shard intersection order $(S_n,\leq)$ is given by descent number: $\rk(w) = d(w)$.
\end{obs}

\subsubsection{Noncrossing partitions}\label{sec:noncrossing}

Reading \cite{R} shows generally that the lattice of noncrossing partitions of type $W$ is an induced sublattice of $(W, \leq)$. For $W=S_n$, this fact can be realized by restricting to the set of \emph{$231$-avoiding permutations}, which are well-known to be in bijection with classical noncrossing partitions.

We say $w \in S_n$ is $231$-avoiding if there is no triple of indices $i < j < k$ such that $w(k) < w(i) < w(j)$. Let $S_n(231)$ denote the set of $231$-avoiding permutations. For example, $51243 \in S_5(231)$ and $31524 \notin S_5(231)$. 

A \emph{noncrossing partition} $\pi = \{ R_1, R_2,\ldots, R_k\}$, is a set partition of $\{1,2,\ldots,n\}$, such that if $\{a, c\} \subseteq R_i$ and $\{b,d\} \subseteq R_j$, with $1\leq a < b < c < d\leq n$, then $i=j$. That is, two pairs of numbers from distinct blocks cannot be interleaved.  Let $NC(n)$ denote the set of all noncrossing partitions of $\{1,2,\ldots,n\}$. For example, $\{ \{1,5\}, \{2\}, \{3,4\}\} \in NC(5)$, while $\{ \{1,3\}, \{2,5\}, \{4\}\} \notin NC(5)$. 

The lattice of noncrossing partitions is the partially ordered set $(NC(n),\leq)$ with $\sigma\leq \tau$ if $\sigma$ refines $\tau$ as a set partition.

Define a bijection $\phi: S_n(231) \to NC(n)$ by mapping the decreasing runs of a permutation to blocks in a partition. See Figure \ref{fig:noncrossing}. Specifically, if $w = d_1| d_2 |\cdots |d_k$, where the $d_i$ are the blocks of decreasing runs of $w$, then letting $D_i$ denote the set of letters of $d_i$, we have $\phi(w) = \{ D_1, D_2,\ldots, D_k\}$.  For example, if $w = 421|3|765|98$, then \[\phi(w) = \{ \{1,2,4\}, \{3\}, \{5, 6, 7\}, \{8,9\}\}.\] The inverse map takes the blocks of $\pi$, lists each block in decreasing order, and then orders the blocks from left to right according to the smallest element in the block. For example, if \[\pi = \{ \{1,7,9\}, \{2,3\}, \{4,6\}, \{5\}, \{8\} \},\] then $\phi^{-1}(\pi) = 971|32|64|5|8.$

\begin{figure}[h]
\[
\begin{tikzpicture}[>=stealth',bend angle=45,auto]
\tikzstyle{state1}=[rectangle,draw=black,scale=.75]
\tikzstyle{state3}=[circle]
\draw (0,0) node[state1] (a) 
 {\begin{tikzpicture}
 \tikzstyle{state2}=[circle,draw=black,scale=1.5]
 \draw (1,9) node[state] (1) {$9$};
 \draw (2,7) node[state] (2) {$7$};
 \draw (3,1) node[state] (3) {$1$};
 \draw (4,3) node[state] (4) {$3$};
 \draw (5,2) node[state] (5) {$2$};
 \draw (6,6) node[state] (6) {$6$};
 \draw (7,4) node[state] (7) {$4$};
 \draw (8,5) node[state] (8) {$5$};
 \draw (9,8) node[state] (9) {$8$};
 \draw[line width=4] (1)--(2);
 \draw[line width=4] (2)--(3);
 \draw[line width=4] (4)--(5);
 \draw[line width=4] (6)--(7);
 \draw[dashed, ->] (1.5,8)--(9);
 \draw[dashed, ->] (2.25,6)--(6);
 \draw[dashed, ->] (2.75,3)--(4);
 \draw[dashed, ->] (6.5,5)--(8);
 \end{tikzpicture}};
\draw (5,0) node[scale=.75] (b) 
 {\begin{tikzpicture}
 \draw (0,1) node[left] (1) {$1$};
 \draw (0,2) node[left] (2) {$2$};
 \draw (0,3) node[left] (3) {$3$};
 \draw (0,4) node[left] (4) {$4$};
 \draw (0,5) node[left] (5) {$5$};
 \draw (0,6) node[left] (6) {$6$};
 \draw (0,7) node[left] (7) {$7$};
 \draw (0,8) node[left] (8) {$8$};
 \draw (0,9) node[left] (9) {$9$};
 \draw [line width=2] (9) to [out=0,in=0] (7);
 \draw [line width=2] (7) to [out=0,in=0] (1);
 \draw [line width=2] (6) to [out=0,in=0] (4);
 \draw [line width=2] (3) to [out=0,in=0] (2);
 \end{tikzpicture}};
\end{tikzpicture}
\]
\caption{The decreasing runs of a 231-avoiding permutation form a noncrossing partition.}\label{fig:noncrossing}
\end{figure}
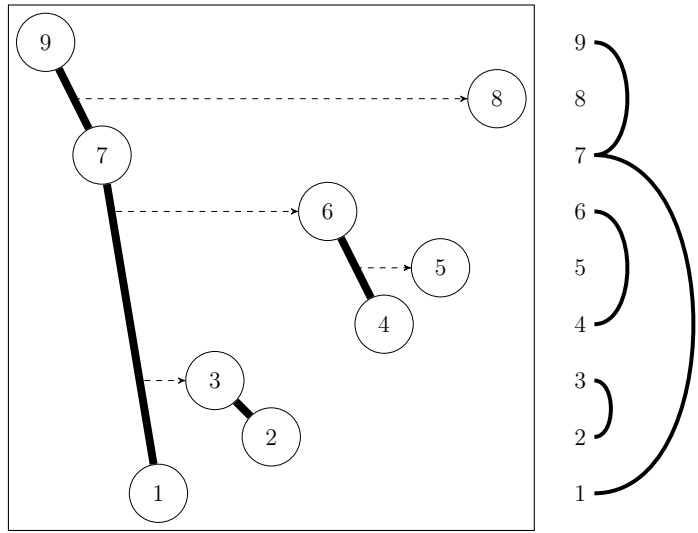

It is straightforward from the definitions that if $u \leq v$ in $(S_n, \leq)$, then $\phi(u) \leq \phi(v)$ in $(NC(n),\leq)$. It is only slightly more subtle to check that the converse is true. We have the following. 

\begin{prp}[\cite{R}, Theorem 8.5]\label{prp:sub}
The lattice of noncrossing partitions $(NC(n),\leq)$ is isomorphic to the induced sublattice $(S_n(231), \leq)$ inside $(S_n, \leq)$.
\end{prp}

We note that in $(NC(n),\leq)$ a noncrossing partition $\pi$ with $k$ blocks has rank $n-1-k$. Since the number of blocks of $\pi$ corresponds to the number of ascents of $w=\phi^{-1}(\pi)$, $\rk(\pi) = d(w)$. Thus we find $(S_n(231),\leq)$ is ranked, with rank given by descent number.

In Figure \ref{fig:intposet}, we see the shard intersection order $(S_4,\leq)$, with the lattice of noncrossing partitions $(NC(4),\leq)\cong (S_n(231),\leq)$ highlighted in bold.

\begin{figure}[ht]
\begin{tikzpicture}[>=stealth',bend angle=45,auto, scale=.8]
\tikzstyle{state1}=[rectangle,draw=black,scale=.33]
\tikzstyle{state4}=[rectangle,draw=black,ultra thick,rounded corners,scale=.33]
\tikzstyle{state3}=[circle]
\draw (0,0) node[state4] (a0) 
 {\begin{tikzpicture}
 \tikzstyle{state2}=[circle,draw=black,scale=1.5]
 \draw (1,1) node[state2] (1) {$1$};
 \draw (2,2) node[state2] (2) {$2$};
 \draw (3,3) node[state2] (3) {$3$};
 \draw (4,4) node[state2] (4) {$4$};
 \end{tikzpicture}};
\draw (4,10) node[state4] (a1)
 {\begin{tikzpicture}
 \tikzstyle{state2}=[circle,draw=black, scale=1.5]
 \draw (1,2) node[state2] (1) {$2$};
 \draw (2,1) node[state2] (2) {$1$};
 \draw (3,3) node[state2] (3) {$3$};
 \draw (4,4) node[state2] (4) {$4$};
 \draw [line width=4] (1)--(2);
 \end{tikzpicture}};
\draw (4,8) node[state4] (a2)
 {\begin{tikzpicture}
 \tikzstyle{state2}=[circle,draw=black,scale=1.5]
 \draw (1,3) node[state2] (1) {$3$};
 \draw (2,1) node[state2] (2) {$1$};
 \draw (3,2) node[state2] (3) {$2$};
 \draw (4,4) node[state2] (4) {$4$};
 \draw [line width=4] (1)--(2);
 \draw [thick,dashed,->] (1.5,2)--(3);
 \end{tikzpicture}};
\draw (4,6) node[state1] (a3)
 {\begin{tikzpicture}
 \tikzstyle{state2}=[circle,draw=black,scale=1.5]
 \draw (1,2) node[state2] (1) {$2$};
 \draw (2,3) node[state2] (2) {$3$};
 \draw (3,1) node[state2] (3) {$1$};
 \draw (4,4) node[state2] (4) {$4$};
 \draw [line width=4] (2)--(3);
 \draw [thick,dashed,->] (1)--(2.5,2);
 \end{tikzpicture}};
\draw (4,4) node[state4] (a4)
 {\begin{tikzpicture}
 \tikzstyle{state2}=[circle,draw=black,scale=1.5]
 \draw (1,4) node[state2] (1) {$4$};
 \draw (2,1) node[state2] (2) {$1$};
 \draw (3,2) node[state2] (3) {$2$};
 \draw (4,3) node[state2] (4) {$3$};
 \draw [line width=4] (1)--(2);
 \draw [thick,dashed,->] (1.2,3)--(4);
 \draw [thick,dashed,->] (1.6,2)--(3);
 \end{tikzpicture}};
\draw (4,2) node[state1] (a5)
 {\begin{tikzpicture}
 \tikzstyle{state2}=[circle,draw=black,scale=1.5]
 \draw (1,2) node[state2] (1) {$2$};
 \draw (2,4) node[state2] (2) {$4$};
 \draw (3,1) node[state2] (3) {$1$};
 \draw (4,3) node[state2] (4) {$3$};
 \draw [line width=4] (2)--(3);
 \draw [thick,dashed,->] (1)--(2.6,2);
 \draw [thick,dashed,->] (2.3,3)--(4);
 \end{tikzpicture}}; 
\draw (4,0) node[state1] (a6)
 {\begin{tikzpicture}
 \tikzstyle{state2}=[circle,draw=black,scale=1.5]
 \draw (1,3) node[state2] (1) {$3$};
 \draw (2,4) node[state2] (2) {$4$};
 \draw (3,1) node[state2] (3) {$1$};
 \draw (4,2) node[state2] (4) {$2$};
 \draw [line width=4] (2)--(3);
 \draw [thick,dashed,->] (1)--(2.3,3);
 \draw [thick,dashed,->] (2.6,2)--(4);
 \end{tikzpicture}}; 
\draw (4,-2) node[state1] (a7)
 {\begin{tikzpicture}
 \tikzstyle{state2}=[circle,draw=black,scale=1.5]
 \draw (1,2) node[state2] (1) {$2$};
 \draw (2,3) node[state2] (2) {$3$};
 \draw (3,4) node[state2] (3) {$4$};
 \draw (4,1) node[state2] (4) {$1$};
 \draw [line width=4] (3)--(4);
 \draw [thick,dashed,->] (2)--(3.2,3);
 \draw [thick,dashed,->] (1)--(3.6,2);
 \end{tikzpicture}}; 
\draw (4,-4) node[state4] (a8)
 {\begin{tikzpicture}
 \tikzstyle{state2}=[circle,draw=black,scale=1.5]
 \draw (1,1) node[state2] (1) {$1$};
 \draw (2,3) node[state2] (2) {$3$};
 \draw (3,2) node[state2] (3) {$2$};
 \draw (4,4) node[state2] (4) {$4$};
 \draw [line width=4] (2)--(3);
 \end{tikzpicture}}; 
\draw (4,-6) node[state4] (a9)
 {\begin{tikzpicture}
 \tikzstyle{state2}=[circle,draw=black,scale=1.5]
 \draw (1,1) node[state2] (1) {$1$};
 \draw (2,4) node[state2] (2) {$4$};
 \draw (3,2) node[state2] (3) {$2$};
 \draw (4,3) node[state2] (4) {$3$};
 \draw [line width=4] (2)--(3);
 \draw [thick,dashed,->] (2.5,3)--(4);
 \end{tikzpicture}}; 
\draw (4,-8) node[state1] (a10)
 {\begin{tikzpicture}
 \tikzstyle{state2}=[circle,draw=black,scale=1.5]
 \draw (1,1) node[state2] (1) {$1$};
 \draw (2,3) node[state2] (2) {$3$};
 \draw (3,4) node[state2] (3) {$4$};
 \draw (4,2) node[state2] (4) {$2$};
 \draw [line width=4] (3)--(4);
 \draw [thick,dashed,->] (2)--(3.5,3);
 \end{tikzpicture}};   
\draw (4,-10) node[state4] (a11)
 {\begin{tikzpicture}
 \tikzstyle{state2}=[circle,draw=black,scale=1.5]
 \draw (1,1) node[state2] (1) {$1$};
 \draw (2,2) node[state2] (2) {$2$};
 \draw (3,4) node[state2] (3) {$4$};
 \draw (4,3) node[state2] (4) {$3$};
 \draw [line width=4] (3)--(4);
 \end{tikzpicture}}; 
\draw (10,10) node[state4] (b1)
 {\begin{tikzpicture}
 \tikzstyle{state2}=[circle,draw=black,scale=1.5]
 \draw (1,3) node[state2] (1) {$3$};
 \draw (2,2) node[state2] (2) {$2$};
 \draw (3,1) node[state2] (3) {$1$};
 \draw (4,4) node[state2] (4) {$4$};
 \draw [line width=4] (1)--(2);
 \draw [line width=4] (2)--(3);
 \end{tikzpicture}}; 
\draw (10,8) node[state4] (b2)
 {\begin{tikzpicture}
 \tikzstyle{state2}=[circle,draw=black,scale=1.5]
 \draw (1,4) node[state2] (1) {$4$};
 \draw (2,2) node[state2] (2) {$2$};
 \draw (3,1) node[state2] (3) {$1$};
 \draw (4,3) node[state2] (4) {$3$};
 \draw [line width=4] (1)--(2);
 \draw [line width=4] (2)--(3);
 \draw [thick,dashed,->] (1.5,3)--(4);
 \end{tikzpicture}};
\draw (10,6) node[state1] (b3)
 {\begin{tikzpicture}
 \tikzstyle{state2}=[circle,draw=black,scale=1.5]
 \draw (1,3) node[state2] (1) {$3$};
 \draw (2,4) node[state2] (2) {$4$};
 \draw (3,2) node[state2] (3) {$2$};
 \draw (4,1) node[state2] (4) {$1$};
 \draw [line width=4] (3)--(4);
 \draw [line width=4] (2)--(3);
 \draw [thick,dashed,->] (1)--(2.5,3);
 \end{tikzpicture}};
\draw (10,4) node[state4] (b4)
 {\begin{tikzpicture}
 \tikzstyle{state2}=[circle,draw=black,scale=1.5]
 \draw (1,2) node[state2] (1) {$2$};
 \draw (2,1) node[state2] (2) {$1$};
 \draw (3,4) node[state2] (3) {$4$};
 \draw (4,3) node[state2] (4) {$3$};
 \draw [line width=4] (1)--(2);
 \draw [line width=4] (3)--(4);
 \end{tikzpicture}};  
\draw (10,2) node[state1] (b5)
 {\begin{tikzpicture}
 \tikzstyle{state2}=[circle,draw=black,scale=1.5]
 \draw (1,4) node[state2] (1) {$4$};
 \draw (2,2) node[state2] (2) {$2$};
 \draw (3,3) node[state2] (3) {$3$};
 \draw (4,1) node[state2] (4) {$1$};
 \draw [line width=4] (1)--(2);
 \draw [line width=4] (3)--(4);
 \draw [thick,dashed,->] (1.5,3)--(3);
 \end{tikzpicture}};
\draw (10,0) node[state1] (b6)
 {\begin{tikzpicture}
 \tikzstyle{state2}=[circle,draw=black,scale=1.5]
 \draw (1,3) node[state2] (1) {$3$};
 \draw (2,1) node[state2] (2) {$1$};
 \draw (3,4) node[state2] (3) {$4$};
 \draw (4,2) node[state2] (4) {$2$};
 \draw [line width=4] (1)--(2);
 \draw [line width=4] (3)--(4);
 \draw [thick,dashed,->] (1.5,2)--(4);
 \end{tikzpicture}};
\draw (10,-2) node[state1] (b7)
 {\begin{tikzpicture}
 \tikzstyle{state2}=[circle,draw=black,scale=1.5]
 \draw (1,2) node[state2] (1) {$2$};
 \draw (2,4) node[state2] (2) {$4$};
 \draw (3,3) node[state2] (3) {$3$};
 \draw (4,1) node[state2] (4) {$1$};
 \draw [line width=4] (3)--(4);
 \draw [line width=4] (2)--(3);
 \draw [thick,dashed,->] (1)--(3.5,2);
 \end{tikzpicture}};
\draw (10,-4) node[state4] (b8)
 {\begin{tikzpicture}
 \tikzstyle{state2}=[circle,draw=black,scale=1.5]
 \draw (1,4) node[state2] (1) {$4$};
 \draw (2,3) node[state2] (2) {$3$};
 \draw (3,1) node[state2] (3) {$1$};
 \draw (4,2) node[state2] (4) {$2$};
 \draw [line width=4] (1)--(2);
 \draw [line width=4] (2)--(3);
 \draw [thick,dashed,->] (2.5,2)--(4);
 \end{tikzpicture}}; 
\draw (10,-6) node[state4] (b9)
 {\begin{tikzpicture}
 \tikzstyle{state2}=[circle,draw=black,scale=1.5]
 \draw (1,4) node[state2] (1) {$4$};
 \draw (2,1) node[state2] (2) {$1$};
 \draw (3,3) node[state2] (3) {$3$};
 \draw (4,2) node[state2] (4) {$2$};
 \draw [line width=4] (1)--(2);
 \draw [line width=4] (3)--(4);
 \draw [thick,dashed,->] (1.3,3)--(3);
 \end{tikzpicture}};
\draw (10,-8) node[state1] (b10)
 {\begin{tikzpicture}
 \tikzstyle{state2}=[circle,draw=black,scale=1.5]
 \draw (1,3) node[state2] (1) {$3$};
 \draw (2,2) node[state2] (2) {$2$};
 \draw (3,4) node[state2] (3) {$4$};
 \draw (4,1) node[state2] (4) {$1$};
 \draw [line width=4] (1)--(2);
 \draw [line width=4] (3)--(4);
 \draw [thick,dashed,->] (2)--(3.6,2);
 \end{tikzpicture}};
\draw (10,-10) node[state4] (b11)
 {\begin{tikzpicture}
 \tikzstyle{state2}=[circle,draw=black,scale=1.5]
 \draw (1,1) node[state2] (1) {$1$};
 \draw (2,4) node[state2] (2) {$4$};
 \draw (3,3) node[state2] (3) {$3$};
 \draw (4,2) node[state2] (4) {$2$};
 \draw [line width=4] (3)--(4);
 \draw [line width=4] (2)--(3);
 \end{tikzpicture}};
\draw (14,0) node[state4] (b0)
 {\begin{tikzpicture}
 \tikzstyle{state2}=[circle,draw=black,scale=1.5]
 \draw (1,4) node[state2] (1) {$4$};
 \draw (2,3) node[state2] (2) {$3$};
 \draw (3,2) node[state2] (3) {$2$};
 \draw (4,1) node[state2] (4) {$1$};
 \draw [line width=4] (1)--(2);
 \draw [line width=4] (2)--(3);
 \draw [line width=4] (3)--(4);
 \end{tikzpicture}};
\draw [line width=2] (a0) to [out=90,in=200] (a1);
\draw [line width=2] (a0) to [out=85,in=195] (a2);
\draw (a0) to [out=80,in=190] (a3);
\draw [line width=2] (a0) to [out=75,in=185] (a4);
\draw (a0) to [out=70,in=180] (a5); 
\draw (a0)--(a6);
\draw (a0) to [out=-70,in=180] (a7);
\draw [line width=2] (a0) to [out=-75,in=175] (a8);
\draw [line width=2] (a0) to [out=-80,in=170] (a9);
\draw (a0) to [out=-85,in=165] (a10);
\draw [line width=2] (a0) to [out=-90,in=160] (a11); 
\draw [line width=2] (b0) to [out=90,in=-20] (b1);
\draw [line width=2] (b0) to [out=95,in=-15] (b2);
\draw (b0) to [out=100,in=-10] (b3);
\draw [line width=2] (b0) to [out=105,in=-5] (b4);
\draw (b0) to [out=110,in=0] (b5); 
\draw (b0)--(b6);
\draw (b0) to [out=-110,in=0] (b7);
\draw [line width=2] (b0) to [out=-105,in=5] (b8);
\draw [line width=2] (b0) to [out=-100,in=10] (b9);
\draw (b0) to [out=-95,in=15] (b10);
\draw [line width=2] (b0) to [out=-90,in=20] (b11);
\draw [line width=2] (5,10)--(9,10);
\draw [line width=2] (5,10)--(9,8);
\draw (5,10)--(9,6);
\draw [line width=2] (5,10)--(9,4);
\draw [line width=2] (5,8)--(9,10);
\draw [line width=2] (5,8)--(9,-4);
\draw (5,8)--(9,0);
\draw (5,6)--(9,10);
\draw (5,6)--(9,2);
\draw (5,6)--(9,-2);
\draw [line width=2] (5,4)--(9,8);
\draw [line width=2] (5,4)--(9,-4);
\draw [line width=2] (5,4)--(9,-6);
\draw (5,2)--(9,8);
\draw (5,2)--(9,-2);
\draw (5,0)--(9,-4);
\draw (5,0)--(9,6);
\draw (5,-2)--(9,-8);
\draw (5,-2)--(9,-2);
\draw (5,-2)--(9,6);
\draw [line width=2] (5,-4)--(9,10);
\draw [line width=2] (5,-4)--(9,-10);
\draw [line width=2] (5,-4)--(9,-6);
\draw (5,-4)--(9,-8);
\draw [line width=2] (5,-6)--(9,8);
\draw [line width=2] (5,-6)--(9,-10);
\draw (5,-6)--(9,2);
\draw (5,-8)--(9,-10);
\draw (5,-8)--(9,0);
\draw (5,-8)--(9,6);
\draw [line width=2] (5,-10)--(9,-10);
\draw [line width=2] (5,-10)--(9,4);
\draw (5,-10)--(9,-2);
\draw [line width=2] (5,-10)--(9,-4);      
\end{tikzpicture}
\caption{The shard intersection lattice for $S_4$ contains the lattice of noncrossing partitions.}\label{fig:intposet}
\end{figure}
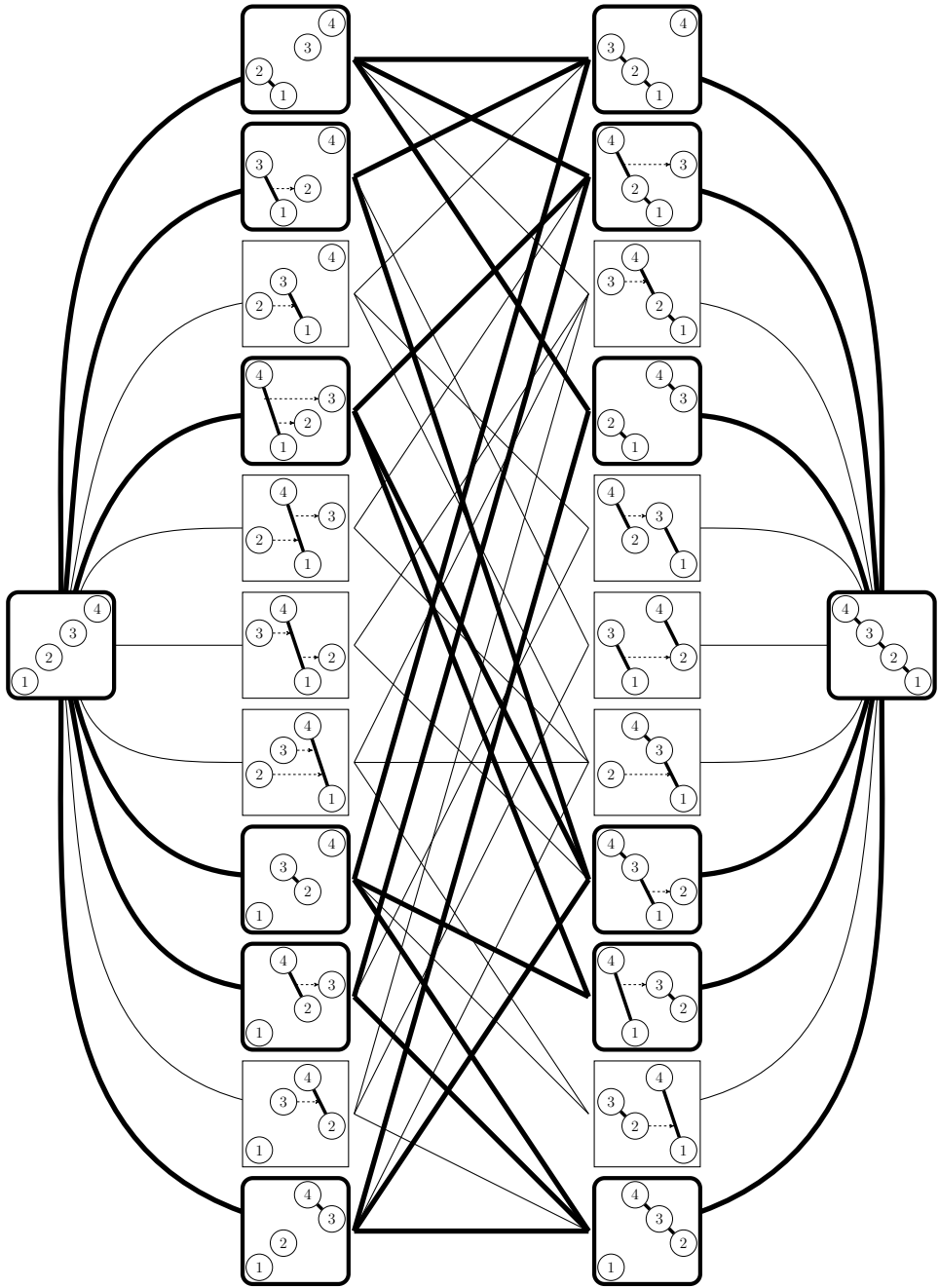

\subsection{Type $B_n$}

The root system of type $B_n$ lives in $V=\R^n$, with positive roots \[ \Phi^+ = \{ \varepsilon_j \pm \varepsilon_i : 1\leq i < j \leq n\} \cup \{ \varepsilon_i : 1\leq i \leq n\}.\]
(The root system of type $C_n$ is simply a rescaling of the $B_n$ root system. Thus, the hyperplane arrangement for $C_n$ is identical to that of $B_n$ and all results that follow in this section hold for  the Coxeter groups of type $C_n$ as well.)

With respect to this choice of root system, the base region $B$ is given by: \[ B = \{ (x_1,\ldots,x_n) \in \R^n : 0< x_1 < \cdots < x_n \}.\] The hyperplane corresponding to the positive root $\varepsilon_j - \varepsilon_i$ is: \[ H_{i,j} = \{ (x_1, \ldots, x_n) \in \R^n : x_i = x_j\},\]
the hyperplane corresponding to the positive root $\varepsilon_j + \varepsilon_i$ is: \[ H_{i,-j} = \{ (x_1,\ldots, x_n)\in \R^n : x_i = -x_j\},\]
and the hyperplane corresponding to the positive root $\varepsilon_i$ is: \[ H_{0,i} = \{(x_1,\ldots,x_n) \in \R^n : x_i = 0\}.\]

There are three possibilities for rank two subarrangements of $\mathcal{A}(B_n)$. The subarrangements are either isomorphic to $\mathcal{A}(A_1\times A_1)$, to $\mathcal{A}(A_2)$, or to $\mathcal{A}(B_2)$. The possibilities are shown in Figures \ref{fig:Asubs} and \ref{fig:Bsubs}.

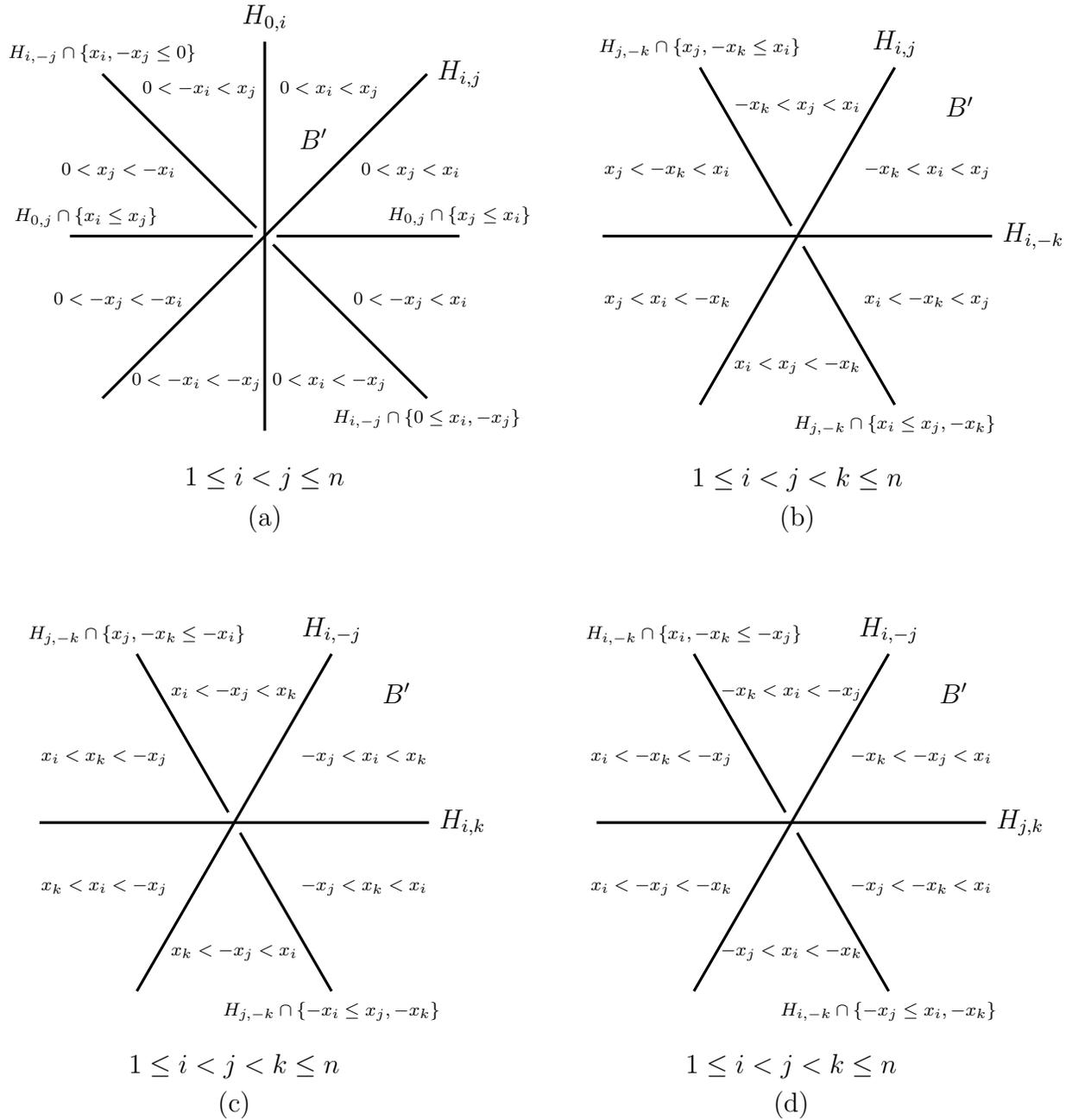
\begin{figure}
\begin{tikzpicture}
\draw[very thick] (-3,0) -- (3,0) node[above] {\tiny $H_{0,j} \cap \{ x_j \leq x_i \}$};
\draw (-2.75,0) node[above] {\tiny $H_{0,j} \cap \{ x_i \leq x_j \}$};
\draw[very thick] (-2.5,2.5) -- (2.5,-2.5) node[below] {\tiny $H_{i,-j} \cap \{ 0\leq x_i, -x_j \}$};
\draw[very thick] (-2.5,2.5) node[above] {\tiny $H_{i,-j} \cap \{ x_i, -x_j \leq 0\}$};
\draw[color=white,fill=white] (0,0) circle (5pt);
\draw[very thick] (0,-3) -- (0,3) node[above] {$H_{0,i}$};
\draw[very thick] (-2.5,-2.5) -- (2.5,2.5) node[right] {$H_{i,j}$};
\draw (.75,1.5) node {$B'$};
\draw (1,2.25) node {\tiny $0 < x_i < x_j$};
\draw (-1,2.25) node {\tiny $0 < -x_i < x_j$};
\draw (2.25,1) node {\tiny $0 < x_j < x_i$};
\draw (2.25,-1) node {\tiny $0 < -x_j < x_i$};
\draw (1,-2.25) node {\tiny $0 < x_i < -x_j$};
\draw (-1.05,-2.25) node {\tiny $0 < -x_i < -x_j$};
\draw (-2.25,-1) node {\tiny $0 < -x_j < -x_i$};
\draw (-2.25,1) node {\tiny $0 < x_j < -x_i$};
\draw (0,-3.75) node {$1\leq i < j \leq n$};
\draw (0,-4) node[below] {(a)};
\end{tikzpicture}
\hspace{.5cm}
\begin{tikzpicture}
\draw[very thick] (-1.5,2.6) -- (1.5,-2.6) node[below] {\tiny $H_{j,-k}\cap \{ x_i \leq x_j, - x_k\}$};
\draw (-1.5,2.6) node[above] {\tiny $H_{j,-k}\cap \{ x_j, -x_k \leq x_i\}$};
\draw[fill=white, color=white] (0,0) circle (5pt);
\draw[very thick] (-3,0) -- (3,0) node[right] {$H_{i,-k}$};
\draw[very thick] (-1.5,-2.6) -- (1.5,2.6) node[above] {$H_{i,j}$};
\draw (2.5,2) node {$B'$};
\draw (2,1) node {\tiny $-x_k < x_i < x_j$};
\draw (2,-1) node {\tiny $x_i < -x_k < x_j$};
\draw (0,2) node {\tiny $-x_k < x_j < x_i$};
\draw (0,-2) node {\tiny $x_i< x_j < -x_k$};
\draw (-2,1) node {\tiny $x_j < -x_k < x_i$};
\draw (-2,-1) node {\tiny $x_j < x_i < -x_k$};
\draw (0,-3.75) node {$1\leq i < j < k \leq n$};
\draw (0,-4) node[below] {(b)};
\end{tikzpicture}

\vspace{1cm}

\begin{tikzpicture}
\draw[very thick] (-1.5,2.6) -- (1.5,-2.6) node[below] {\tiny $H_{j,-k}\cap \{ -x_i \leq x_j, - x_k\}$};
\draw (-1.5,2.6) node[above] {\tiny $H_{j,-k}\cap \{ x_j, -x_k \leq -x_i\}$};
\draw[fill=white, color=white] (0,0) circle (5pt);
\draw[very thick] (-3,0) -- (3,0) node[right] {$H_{i,k}$};
\draw[very thick] (-1.5,-2.6) -- (1.5,2.6) node[above] {$H_{i,-j}$};
\draw (2.5,2) node {$B'$};
\draw (2,1) node {\tiny $-x_j < x_i < x_k$};
\draw (2,-1) node {\tiny $-x_j < x_k < x_i$};
\draw (0,2) node {\tiny $x_i < -x_j < x_k$};
\draw (0,-2) node {\tiny $x_k< -x_j < x_i$};
\draw (-2,1) node {\tiny $x_i < x_k < -x_j$};
\draw (-2,-1) node {\tiny $x_k < x_i < -x_j$};
\draw (0,-3.75) node {$1\leq i < j < k \leq n$};
\draw (0,-4) node[below] {(c)};
\end{tikzpicture}
\hspace{1cm}
\begin{tikzpicture}
\draw[very thick] (-1.5,2.6) -- (1.5,-2.6) node[below] {\tiny $H_{i,-k}\cap \{ -x_j \leq x_i, - x_k\}$};
\draw (-1.5,2.6) node[above] {\tiny $H_{i,-k}\cap \{ x_i, -x_k \leq -x_j\}$};
\draw[fill=white, color=white] (0,0) circle (5pt);
\draw[very thick] (-3,0) -- (3,0) node[right] {$H_{j,k}$};
\draw[very thick] (-1.5,-2.6) -- (1.5,2.6) node[above] {$H_{i,-j}$};
\draw (2.5,2) node {$B'$};
\draw (2,1) node {\tiny $-x_k < -x_j < x_i$};
\draw (2,-1) node {\tiny $-x_j < -x_k < x_i$};
\draw (0,2) node {\tiny $-x_k < x_i < -x_j$};
\draw (0,-2) node {\tiny $-x_j< x_i < -x_k$};
\draw (-2,1) node {\tiny $x_i < -x_k < -x_j$};
\draw (-2,-1) node {\tiny $x_i < -x_j < -x_k$};
\draw (0,-3.75) node {$1\leq i < j < k \leq n$};
\draw (0,-4) node[below] {(d)};
\end{tikzpicture}
\caption{The rank two subarrangements of $\mathcal{A}(B_n)$ not pictured in Figure \ref{fig:Asubs}.}\label{fig:Bsubs}
\end{figure}

As we will show, the cutting relations for hyperplanes $H_{i,j}$, with $0\leq i < j \leq n$, are rather different from the cutting relations for the hyperplanes $H_{i,-j}$, with $1\leq i < j \leq n$.

It is easy to see that a hyperplane $H_{i,j}$, with $0\leq i < j \leq n$ is either cut according to the arrangement $\mathcal{A}(A_2)$ in Figure \ref{fig:Asubs} or, if $i=0$, according to Figure \ref{fig:Bsubs} (a). In either case, we find the shards of $H_{i,j}$ are formed by choosing, for each $k$ such that $i < k < j$, whether $x_k \leq x_i = x_j$ or $x_i = x_j \leq x_k$. In particular, there are $2^{j-i-1}$ shards of this hyperplane, just as we found in Proposition \ref{prp:Ashards} for type $A_{n-1}$.

We now turn to hyperplanes of the form $H_{i,-j}$. The cutting relations for these hyperplanes appear in the arrangements of Figure \ref{fig:Bsubs} (a), (b), (c), and (d).

In case (a), we have two choices. Either $0\leq x_i = -x_j$, or $x_i=-x_j \leq 0$.

Now consider cases (b) and (c). Suppose, without loss of generality, that $0 \leq x_i = -x_j$. Here we need to choose, for each $k$ such that $1\leq k < i$, whether:
\begin{itemize}
\item $-x_k \leq -x_i = x_j \leq 0 \leq x_i = -x_j \leq x_k$,
\item $-x_i = x_j \leq -x_k, 0, x_k \leq x_i =-x_j$, or
\item $x_k \leq -x_i = x_j \leq 0 \leq x_i =-x_j \leq -x_k$.
\end{itemize}
Note that we could \emph{not} have $0 \leq x_i = -x_j \leq -x_k, x_k$, as all coordinates would be forced to equal zero. Hence there are three choices for each such $k$, yielding a total of $3^{i-1}$ choices of this kind.

Finally consider case (d). Here we see we need to choose, for each $k$ such that $i < k < j$, whether $x_k \leq x_i = -x_j$ or $x_i = -x_j \leq x_k$, yielding $2^{j-i-1}$ choices.

We have now completely described the shards of type $B_n$. In particular, we have the following Proposition. (The formula for the sum is easily verified by induction.)

\begin{prp}\label{prp:Bshards}
For all $0\leq i < j \leq n$, the hyperplane $H_{i,j}$ has $2^{j-i-1}$ shards. For any $1\leq i < j \leq n$, the hyperplane $H_{i,-j}$ has $2^{j-i}3^{i-1}$ shards. Therefore, there are \[ \sum_{0\leq i < j\leq n} 2^{j-i-1} + \sum_{1\leq i < j \leq n} 2^{j-i}3^{i-1} = 3^n - n-1\] shards of $\mathcal{A}(B_n)$ in all.
\end{prp}

We now encode intersections of shards with signed permutations.

The Coxeter group of type $B_n$ is known as the hyperoctahedral group, whose elements are bijections \[w: \{-n,\ldots, -1,0,1,\ldots,n\} \to \{-n,\ldots,-1,0,1,\ldots,n\}\] such that $w(-i) = -w(i)$ for all $i$. In particular $w(0)=0$ for any element $w$.  We write elements in one-line notation: $w=w(-n)\cdots w(-1)0w(1)\cdots w(n)$.  For example $w = \bar3 5 4 \bar2 1 0\bar1 2\bar 4\bar 5 3$ is a signed permutation (with bars instead of minus signs to save space). Note that because of the symmetry condition, a given $w$ is determined by $w(1), \ldots, w(n)$, e.g., it would suffice to say $w = \bar1 2\bar 4\bar 5 3$ above.  

A \emph{descent} of an element $w \in B_n$ is a letter $w(i)$, $i\geq 0$, such that $w(i) > w(i+1)$. We denote the number of type $B_n$ descents by $d_B(w)$. Thus $w = \bar1 2\bar 4\bar 5 3$ has $d_B(w) = 3$. Notice that, as a word, $w(-n)\cdots w(-1) 0 w(1) \cdots w(n)$ has exactly twice as many descents as $w$, since if $w(i) > w(i+1)$, then $w(-i-1)=-w(i+1) > -w(i) = w(-i)$. As with the symmetric group, we will highlight the maximal decreasing runs of $w$, written in long form, by inserting bars in ascent positions. For example, we write \[ w = \bar3 | 5 4 \bar2 | 1 0\bar1 | 2\bar 4\bar 5 | 3.\]

Visually, we represent a signed permutation as an array with a mark in column $i$, row $j$ ($-n\leq i,j \leq n$) if $w(i)=j$. As with the type $A_{n-1}$ model, we group together decreasing runs into blocks indicated by thick lines: 
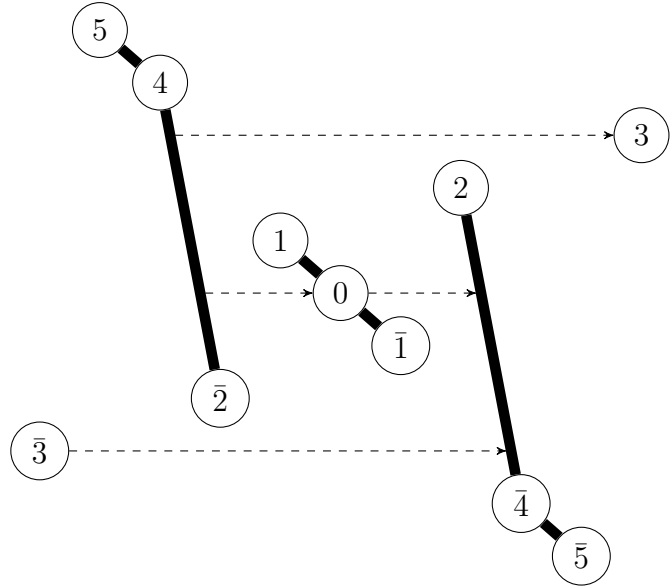
\begin{equation}\label{eq:Bex}
\begin{tikzpicture}[>=stealth',bend angle=45,auto,xscale=.8, yscale=.7]
\tikzstyle{state}=[circle,draw=black,minimum size=4mm]
\draw (-4,5) node[state] (5) {$5$};
\draw (-2,-2) node[state] (2b) {$\bar 2$};
\draw (3,-4) node[state] (4b) {$\bar 4$};
\draw (-5,-3) node[state] (3b) {$\bar 3$};
\draw (-1,1) node[state] (1) {$1$};
\draw (0,0) node[state] (0) {$0$};
\draw (1,-1) node[state] (1b) {$\bar 1$};
\draw (5,3) node[state] (3) {$3$};
\draw (-3,4) node[state] (4) {$4$};
\draw (2,2) node[state] (2) {$2$};
\draw (4,-5) node[state] (5b) {$\bar 5$};
\draw[line width=4] (5)--(4)--(2b);
\draw[line width=4] (1)--(0)--(1b);
\draw[line width=4] (2)--(4b)--(5b);
\draw[dashed, ->] (3b)--(2.75,-3);
\draw[dashed, ->] (-2.25,0)--(0);
\draw[dashed, ->] (0)--(2.25,0);
\draw[dashed, ->] (-2.75,3)--(3);
\end{tikzpicture}
\end{equation}
If it is possible to draw a horizontal line to connect two decreasing runs, the block on the left is considered less than the block on the right. This gives a certain pre-order on $\{0,\pm 1, \ldots, \pm n\}$ that we will call a \emph{signed permutation pre-order}.

We will show that signed permutation pre-orders are in bijection with type $B_n$ shard intersections. Just as with the type $A_{n-1}$ model, we define a cone of points, $C(w)$, for an element $w \in B_n$ as follows:
\begin{itemize}
\item if $i$ and $j$ are in the same block in $w$, then we have $x_i = x_j$, with the understanding that $x_{-i} = -x_i$ and $x_0 = 0$,
\item if $i< k < j$ and $k$ is not in the same block as $i$ and $j$, then:
\begin{itemize}
\item[a)] $x_k \leq x_i = x_j$ if $k$ appears to the left of $i$ in $w$, and 
\item[b)] $x_i = x_j \leq x_k$ if $k$ appears to the right of $i$ in $w$.
\end{itemize}
\end{itemize}

The example shown in \eqref{eq:Bex} then corresponds to the set of points in $\R^5$ satisfying:
\[ x_1 = 0 = -x_1 \leq x_2 = -x_4 = -x_5 \geq -x_3.\]

Each block has a negative counterpart, except for the block containing zero. Thus the dimension of $C(w)$ is half the number of blocks not containing zero, plus one if there is a nonzero number in the block with zero. Thus codimension corresponds to the type $B_n$ descent statistic.

\begin{obs}\label{obs:desB}
For any $w \in B_n$, \[ d_B(w) = n-\dim(C(w)).\] In particular, shards correspond to signed permutations with exactly one type $B_n$ descent.
\end{obs}

Thus Proposition \ref{prp:Bshards} gives an indirect way to prove there are $3^n - n- 1$ signed permutations with exactly one descent.

It is straightforward to prove that such cones always correspond to intersections of type $B_n$ shards, and we get the following analogue of Proposition \ref{prp:Awelldef}.

\begin{prp}\label{prp:Bwelldef}(Signed permutation pre-orders correspond to shard intersections)
Every intersection of $B_n$ shards equals $C(w)$ for some $w\in B_n$, and for any $w \in B_n$, the cone corresponding to $w$ is an intersection of shards.
\end{prp}

%
%
%
%
%

We can define the shard intersection order on $B_n$ just as given in Definition \ref{def:order}. Namely, $u \leq v$ in $(B_n,\leq)$ if and only if $C(v)\subseteq C(u)$. This manifests itself for signed permutation pre-orders in the same notions of ``refinement" and ``consistency" given in Proposition \ref{prp:order}. We can use the same intuition of merging blocks to move up in the poset, taking care to act symmetrically: if $i$ joins a block with $j$, then $-i$ must join a block with $-j$ and so on.

For example, $\bar 4 \bar 5| \bar 3 |\bar 2 | 1 0 \bar 1| 2 | 3| 54 < \bar3 | 5 4 \bar2 | 1 0\bar1 | 2\bar 4\bar 5 | 3 $ as shown:
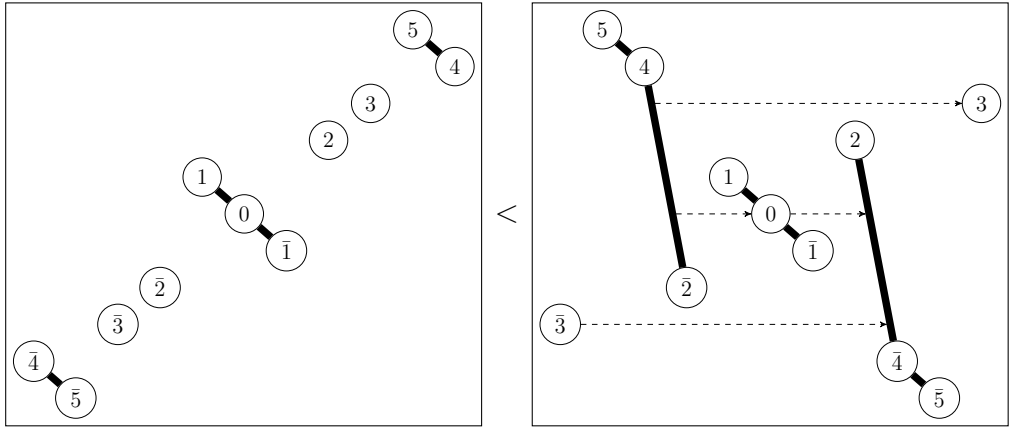
\[
\begin{tikzpicture}
\tikzstyle{state1}=[rectangle,draw=black, scale=.7]
\draw (0,0) node[state1]
 {\begin{tikzpicture}[>=stealth',bend angle=45,auto,xscale=.8, yscale=.7]
\tikzstyle{state}=[circle,draw=black,minimum size=4mm]
\draw (4,5) node[state] (5) {$5$};
\draw (-2,-2) node[state] (2b) {$\bar 2$};
\draw (-5,-4) node[state] (4b) {$\bar 4$};
\draw (-3,-3) node[state] (3b) {$\bar 3$};
\draw (-1,1) node[state] (1) {$1$};
\draw (0,0) node[state] (0) {$0$};
\draw (1,-1) node[state] (1b) {$\bar 1$};
\draw (3,3) node[state] (3) {$3$};
\draw (5,4) node[state] (4) {$4$};
\draw (2,2) node[state] (2) {$2$};
\draw (-4,-5) node[state] (5b) {$\bar 5$};
\draw[line width=4] (5)--(4);
\draw[line width=4] (4b)--(5b);
\draw[line width=4] (1)--(0)--(1b);
\end{tikzpicture}};
\draw (3.5,0) node {$<$};
\draw (7,0) node[state1]
 {\begin{tikzpicture}[>=stealth',bend angle=45,auto,xscale=.8, yscale=.7]
\tikzstyle{state}=[circle,draw=black,minimum size=4mm]
\draw (-4,5) node[state] (5) {$5$};
\draw (-2,-2) node[state] (2b) {$\bar 2$};
\draw (3,-4) node[state] (4b) {$\bar 4$};
\draw (-5,-3) node[state] (3b) {$\bar 3$};
\draw (-1,1) node[state] (1) {$1$};
\draw (0,0) node[state] (0) {$0$};
\draw (1,-1) node[state] (1b) {$\bar 1$};
\draw (5,3) node[state] (3) {$3$};
\draw (-3,4) node[state] (4) {$4$};
\draw (2,2) node[state] (2) {$2$};
\draw (4,-5) node[state] (5b) {$\bar 5$};
\draw[line width=4] (5)--(4)--(2b);
\draw[line width=4] (1)--(0)--(1b);
\draw[line width=4] (2)--(4b)--(5b);
\draw[dashed, ->] (3b)--(2.75,-3);
\draw[dashed, ->] (-2.25,0)--(0);
\draw[dashed, ->] (0)--(2.25,0);
\draw[dashed, ->] (-2.75,3)--(3);
\end{tikzpicture}}; 
\end{tikzpicture}
\]
In moving from the signed permutation on the left to the one on the right, we merged $\bar 2$ with the block $54$ (and hence $2$ with $\bar 4\bar 5$). This meant that we needed to decide whether the new block would be right or left of $3$ and right or left of the block containing $0$. In this case, we chose $54\bar 2$ to be left of both.

The lattice of $B_n$ shard intersections is ranked by codimension, so by Observation \ref{obs:desB} we have the following.

\begin{obs}
The rank of a signed permutation $w$ in the shard intersection order $(B_n,\leq)$ is given by descent number: $\rk(w) = d_B(w)$.
\end{obs}

\subsection{Type $D_n$}

The root system of type $D_n$ lives in $V=\R^n$, with positive roots
\[ \Phi^+ = \{ \varepsilon_j \pm \varepsilon_i : 1 \leq i < j \leq n\}.\]
With respect to this choice, the base region $B$ is given by: \[ B = \{ (x_1,\ldots,x_n) \in \R^n : -x_2 < \pm x_1 < x_2 < \cdots < x_n \}.\]

Notice that the $D_n$ roots are all the $B_n$ roots save the standard basis elements. Thus the arrangement $\mathcal{A}(D_n)$ is the subarrangement of $\mathcal{A}(B_n)$ generated by the hyperplanes $H_{i,j}$ and $H_{i,-j}$ (but not $H_{0,i}$).

The rank two subarrangements of $\mathcal{A}(D_n)$ either look like $\mathcal{A}(A_1\times A_1)$ or like $\mathcal{A}(A_2)$, and we can identify all the cutting relations from the pictures in Figure \ref{fig:Asubs} and Figure \ref{fig:Bsubs} (b), (c), and (d).

The hyperplanes $H_{i,j}$, with $1\leq i < j \leq n$, are once again cut according to relation in Figure \ref{fig:Asubs}. We find $2^{j-i-1}$ shards of this hyperplane as in types $A_{n-1}$ and $B_n$.

Now consider a hyperplane $H_{i,-j}$ with $1\leq i < j\leq n$. The cutting relations for this hyperplane are given by parts (b), (c), and (d) of Figure \ref{fig:Bsubs}. From part (d) we see that for each $k$ such that $i < k < j$, we must choose whether $-x_k \leq x_i = -x_j$ or whether $x_i = -x_j \leq -x_k$, yielding $2^{j-i-1}$ choices.

The interaction between the relations in parts (b) and (c) are somewhat delicate. Since we have no hyperplanes of the form $H_{0,i}$, we do not know explicitly whether $x_i = -x_j$ is weakly positive or negative. However, if we know that, say, $x_i = -x_j \leq \pm x_k$, we can infer that $x_i=-x_j$ is negative. Likewise, if $\pm x_k \leq x_i = -x_j$, we can infer that $x_i = -x_j$ is positive. If $k$ is such that $1\leq k < i$ and both $x_k$ and $-x_k$ are on the same side of $x_i=-x_j$ we say $k$ is in the \emph{zero block} of the shard.

If the zero block is empty, we know that for each $k = 1,\ldots, i-1$, there are two choices:
\begin{itemize}
\item $x_k \leq x_i = -x_j, -x_i = x_j \leq -x_k$, or 
\item $-x_k \leq x_i=-x_j, -x_i = x_j \leq x_k$.
\end{itemize}
Thus there are $2^{j-i-1}\cdot 2^{i-1}$ shards of $H_{i,-j}$ with an empty zero block. Note, however, that $x_i= -x_j$ and $-x_i = x_j$ are incomparable.

We will now count the remaining shards in $H_{i,-j}$ according to the smallest element in the zero block. 

Suppose $h$ is the smallest element in the zero block. First of all, since the zero block is nonempty, we know whether $x_i=-x_j$ is weakly positive or weakly negative, giving two initial choices. Suppose, without loss of generality, that $\pm x_h \leq x_i = -x_j$.

Then for each $g=1,\ldots,h-1$, there are two choices:
\begin{itemize}
\item $x_g \leq -x_i = x_j \leq \pm x_h \leq x_i = -x_j \leq -x_g$, or
\item $-x_g \leq -x_i = x_j \leq \pm x_h \leq x_i = -x_j \leq x_g$.
\end{itemize}
For each $k=h+1,\ldots,i-1$, there are three choices:
\begin{itemize}
\item $x_k \leq -x_i = x_j \leq \pm x_h \leq x_i = -x_j \leq -x_k$,
\item $-x_i = x_j \leq \pm x_h, \pm x_k \leq x_i = -x_j$, or
\item $-x_k \leq -x_i = x_j \leq \pm x_h \leq x_i = -x_j \leq x_k$.
\end{itemize}
Hence, we find a total of $2^{j-i-1}\cdot 2\cdot 2^{h-1}\cdot 3^{i-1-h}$ choices for a given $h$.

Pulling all the cases for the zero block together (empty and $h=1,\ldots,i-1$) we find a total of:
\begin{align*}
 2^{j-i-1}\left(2^{i-1} + 2\cdot 3^{i-2} + \cdots + 2^{i-2}\cdot 3 + 2^{i-1}\right) &= 2^{j-i-1}( 2^{i-1} + 2( 3^{i-1} - 2^{i-1}) ) \\
 &= 2^{j-i-1}( 2\cdot 3^{i-1} - 2^{i-1}) \\
 &= 2^{j-i}\cdot 3^{i-1} - 2^{j-2}
\end{align*} 
shards in $H_{i,-j}$.

We have now characterized the shards of type $D_n$. In particular we have the following companion to Propositions \ref{prp:Ashards} and \ref{prp:Bshards}.

\begin{prp}\label{prp:Dshards}
For all $1\leq i < j \leq n$, the hyperplane $H_{i,j}$ has $2^{j-i-1}$ shards, while the hyperplane $H_{i,-j}$ has $2^{j-i}\cdot 3^{i-1} - 2^{j-2}$ shards. Therefore, there are \[ \sum_{1\leq i < j \leq n} 2^{j-i-1} + 2^{j-i}\cdot3^{i-1}-2^{j-2} = 3^n - n2^{n-1} - n -1\] shards of $\mathcal{A}(D_n)$ in all.
\end{prp}

A common description of the elements $w \in D_n$ is as \emph{even signed permutations}, i.e., elements $w \in B_n$ such that there are an even number of negative numbers among $\{w(1),\ldots,w(n)\}$. Here, we prefer to declare that the sign of $w(1)$ is not known. That is, we know $\{w(1), w(-1) \} = \{ j, -j\}$, but we don't know which is which. We write elements as ``forked" signed permutations, e.g., \begin{equation}\label{eq:fork}
 w = \bar 2 3 1 5 \begin{array}{c}4 \\ \bar 4\end{array} \bar 5 \bar 1 \bar 3 2
\end{equation}
corresponds to $\{w(1),-w(1)\} = \{4,-4\}$, $w(2) = -5$, $w(3)=-1$, $w(4) = -3$, and $w(5) = 2$. As an even signed permutation, we would write $w = \bar 4 \bar 5 \bar 1 \bar 3 2$. We choose the forked model because it is more indicative of the geometry of the corresponding region in the complement of $\mathcal{A}(D_5)$: \[ -x_2 < x_3 < x_1 < x_5 < \pm x_4 < -x_5 < -x_1 < -x_3 < x_2.\]  

A \emph{descent} of an element $w \in D_n$ is a letter $w(i)$, $i\geq -1$, such that $w(i) > w(|i|+1)$. That is, we have the usual notion of descents in $w(1)\cdots w(n)$, along with one more if $w(-1) > w(2)$. Let $d_D(w)$ denote the number of type $D_n$ descents of $w$. For example, $w$ shown in \eqref{eq:fork} has $d_D(w) =3$.

We draw $w \in D_n$ as an array with a mark in column $i\geq 0$, row $j$ if $w(i+1) = j$. We put a mark in column $i \leq 0$, row $j$ if $w(i-1) = j$. In effect, we draw $w$ as if it is a type $B_n$ element, then slide $w(i)$ one step left for $i$ positive, one step right for $i$ negative. Hence, $w(1)$ and $-w(1)$ appear in the same center column. Again, we draw solid lines in descent positions. For example, $w = \bar 2 3 1 5 \begin{array}{c}4 \\ \bar 4\end{array} \bar 5 \bar 1 \bar 3 2$ is drawn as:
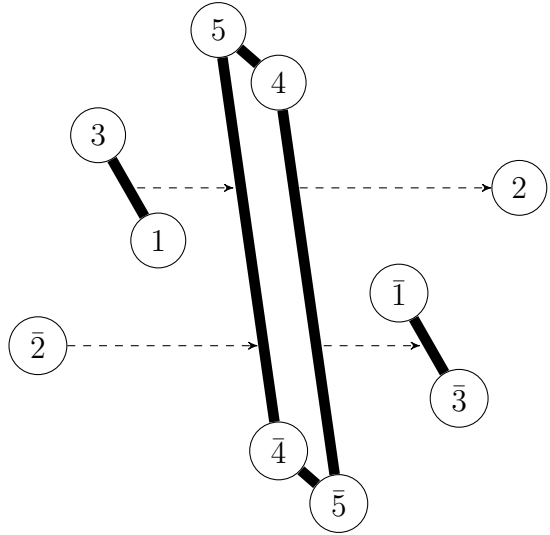
\begin{equation}\label{eq:Dex}
\begin{tikzpicture}[>=stealth',bend angle=45,auto,xscale=.8, yscale=.7]
\tikzstyle{state}=[circle,draw=black,minimum size=4mm]
\draw (-1,5) node[state] (5) {$5$};
\draw (-4,-1) node[state] (2b) {$\bar 2$};
\draw (0,-3) node[state] (4b) {$\bar 4$};
\draw (3,-2) node[state] (3b) {$\bar 3$};
\draw (-2,1) node[state] (1) {$1$};
\draw (2,0) node[state] (1b) {$\bar 1$};
\draw (-3,3) node[state] (3) {$3$};
\draw (0,4) node[state] (4) {$4$};
\draw (4,2) node[state] (2) {$2$};
\draw (1,-4) node[state] (5b) {$\bar 5$};
\draw[line width=4] (5)--(4)--(5b);
\draw[line width=4] (5)--(4b)--(5b);
\draw[line width=4] (3)--(1);
\draw[line width=4] (1b)--(3b);
\draw[dashed, ->] (2b)--(-.35,-1);
\draw[dashed, ->] (-2.35,2)--(-.75,2);
\draw[dashed, ->] (.35,2)--(2);
\draw[dashed, ->] (.75,-1)--(2.35,-1);
\end{tikzpicture}
\end{equation}
The partial order on blocks in this case is similar to earlier cases, with one caveat. Usually, if it is possible to draw a horizontal line to connect two decreasing runs, the block on the left is considered less than the block on the right. However, if $w(1)$ and $w(-1)$ are in distinct blocks, these blocks are only comparable if there is a triple $i < k < j$ with $i,j$ in the block containing $w(1)$ and $k$ in the block containing $w(-1)$. For example, in \eqref{eq:Dex2} the block containing $w(-1)$ and the block containing $w(1)$ are incomparable:
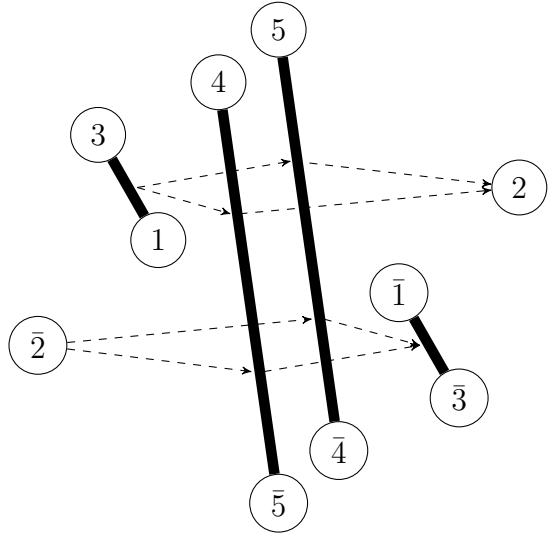
\begin{equation}\label{eq:Dex2}
\begin{tikzpicture}[>=stealth',bend angle=45,auto,xscale=.8, yscale=.7]
\tikzstyle{state}=[circle,draw=black,minimum size=4mm]
\draw (0,5) node[state] (5) {$5$};
\draw (-4,-1) node[state] (2b) {$\bar 2$};
\draw (1,-3) node[state] (4b) {$\bar 4$};
\draw (3,-2) node[state] (3b) {$\bar 3$};
\draw (-2,1) node[state] (1) {$1$};
\draw (2,0) node[state] (1b) {$\bar 1$};
\draw (-3,3) node[state] (3) {$3$};
\draw (-1,4) node[state] (4) {$4$};
\draw (4,2) node[state] (2) {$2$};
\draw (0,-4) node[state] (5b) {$\bar 5$};
\draw[line width=4] (4)--(5b);
\draw[line width=4] (5)--(4b);
\draw[line width=4] (3)--(1);
\draw[line width=4] (1b)--(3b);
\draw[dashed, ->] (2b)--(-.45,-1.5);
\draw[dashed, ->] (2b)--(.55,-.5);
\draw[dashed, ->] (-.25,-1.5)--(2.35,-1);
\draw[dashed, ->] (-2.35,2)--(-.8,1.5);
\draw[dashed, ->] (-2.35,2)--(.2,2.5);
\draw[dashed, ->] (.25,2.5)--(2);
\draw[dashed, ->] (.75,-.5)--(2.35,-1);
\draw[dashed, ->] (-.65,1.5)--(2);
\end{tikzpicture}
\end{equation} 

In either case, we get a pre-order on $\{-n,\ldots,-1,1,\ldots,n\}$, which we call a \emph{forked permutation pre-order}.

For any $w \in D_n$, we define a cone of points $C(w)$ in $\R^n$ just as in the $A_{n-1}$ and $B_n$ cases. Specifically,
\begin{itemize}
\item if $i$ and $j$ are in the same block in $w$, we have $x_i=x_j$, with the understanding that $x_{-i} =-x_i$,
\item if $i < k < j$ and $k$ is not in the same block as $i$ and $j$, then:
\begin{itemize}
\item[a)] $x_k \leq x_i = x_j$ if $k$ is less than $i$ in the pre-order given by $w$, and 
\item[b)] $x_i = x_j \leq x_k$ if $k$ greater than $i$ in the pre-order given by $w$.
\end{itemize}
\end{itemize}

The example shown in \eqref{eq:Dex} then corresponds to the set of points satisfying:
\[ -x_2, x_1 = x_3 \leq x_4=x_5=-x_4=-x_5 (=0) \leq -x_1=-x_3, x_2\] while the example shown in \eqref{eq:Dex2} corresponds to:
\[ -x_2, x_1 = x_3 \leq x_4=-x_5, x_4=-x_5 \leq -x_1=-x_3, x_2 \]

As with earlier cases, we can determine the dimension of $C(w)$ by the number of nonzero blocks and whether there are any coordinates equal to zero. We have the following.

\begin{obs}\label{obs:desD}
For any $w \in D_n$, \[ d_D(w) = n-\dim(C(w)).\] In particular, shards correspond to forked signed permutations with exactly one type $D_n$ descent.
\end{obs}

Thus, Proposition \ref{prp:Dshards} shows there are $3^n -n2^{n-1} -n -1$ elements of $D_n$ with exactly one descent.

That the cones $C(w)$ correspond to intersections of type $D_n$ shards follows from explicit decomposition of a given cone into shards along similar lines as earlier cases, and we have the following.

\begin{prp}
Every intersection of $D_n$ shards equals $C(w)$ for some $w \in D_n$, and for any $w\in D_n$, the cone corresponding to $w$ is an intersection of shards.
\end{prp}

As an example, we show how the forked permutation in \eqref{eq:Dex} can be written as an intersection of $D_n$ shards (with bars drawn to indicate divisions between the blocks):
\[
w = \bar 2 | 3 1 |5 \begin{array}{c}4 \\ \bar 4\end{array} \bar 5| \bar 1 \bar 3 |2 = \bar 4 \bar 5 |\bar 3 |\bar 2 |\begin{array}{c} 1 \\ \bar 1 \end{array}| 2| 3| 5 4 \cap \bar 2 | 1 | 3 | 4 \begin{array}{c} 5 \\ \bar 5 \end{array} \bar 4 | \bar 3 | \bar 1 | 2  \cap \bar 5 | \bar 4 | \bar 2 | \bar 1 \begin{array}{c} 3 \\ \bar 3 \end{array} 1 | 2 | 4 | 5
\]

The shard intersection order on $D_n$ is also analogous to earlier examples. We have $u \leq v$ in $(D_n,\leq)$ if and only if $C(v) \subseteq C(u)$, and the containment of cones can be easily captured by the merging of blocks consistent with the forked pre-order. 

The lattice of $D_n$ shard intersections is ranked by codimension, so by Observation \ref{obs:desD} we have the following.

\begin{obs}
The rank of a forked permutation $w$ in the shard intersection order $(D_n,\leq)$ is given by descent number: $\rk(w) = d_D(w)$. 
\end{obs}

\section{Shellability of the shard intersection orders}\label{sec:shell}

The goal of this section is to prove Theorem \ref{thm:EL}. We will extend Bancroft's EL-labeling for $(S_n,\leq)$ to an EL-labeling for $(B_n,\leq)$; a slight variation allows us to construct an EL-labeling for $(D_n,\leq)$ as well. Hence the order complexes for these posets are lexicographically shellable. 

First recall an \emph{edge labeling}, $\lambda$, of a poset $(P,\leq)$ is an assignment of a label for each edge in the Hasse diagram for $P$. That is, let $\mathcal{E}(P)$ denote the set of cover relations of $P$, denoted $x\prec y$ (i.e., $x < y$ and $x\leq z \leq y$ implies $x=z$ or $z=y$), and let $\Lambda$ be some totally ordered set. Then $\lambda$ is a function $\lambda: \mathcal{E}(P) \to \Lambda$. If $c$ is an unrefinable chain: $x_0 \prec x_1 \prec \cdots \prec x_k$ of elements in $P$, i.e., $x_i \prec x_{i+1}$ is a cover relation for all $i$, then we refer to the label of $c$ as $\lambda(c)=(\lambda(x_0,x_1), \lambda(x_1,x_2),\ldots, \lambda(x_{k-1},x_k))$. We say that a chain $c$ is \emph{rising} if $\lambda(c)$ is a weakly increasing sequence: $\lambda(x_0,x_1)\leq \lambda(x_1,x_2)\leq \cdots \leq \lambda(x_{k-1},x_k)$. Otherwise, if $\lambda(x_{i-1},x_i) > \lambda(x_i,x_{i+1})$ for some $i$, we say $c$ is \emph{has a fall} in $i$. A chain for which each cover is a fall is called a \emph{falling chain}.

\begin{defn}
A labeling $\lambda$ is an \emph{EL-labeling} if, for every interval $[u,v]$ in $P$:
\begin{itemize}
\item there is a unique rising maximal chain $c: u = u_0 \prec u_1 \prec \cdots \prec u_k = v$, and
\item if $d$ is any other maximal chain in $[u,v]$, then $\lambda(c) < \lambda(d)$ in lexicographic order.
\end{itemize}
\end{defn}

We remark that in many treatments, a rising chain is defined to have \emph{strictly} increasing labels, though in Bj\"orner's original work he uses weakly increasing labels \cite{Bj}. Either definition allows one to conclude shellability and so forth.

We now give EL-labelings for the shard intersection orders described earlier. While these labelings take Bancroft's labeling as a starting point and generalize it to other types, what we have is far from a uniform proof of shellability. 

\subsection{Proof of shellability for type $A_{n-1}$}

To describe our EL-labeling, we insert bars in ascent positions of permutations $u = d_1 | d_2 | \cdots | d_k$, so that each word $d_i$ is a maximal decreasing run, and hence a block in the corresponding permutation pre-order.

Given a pair of elements $u<v$ in $(S_n,\leq)$, define the ``merging blocks" of $u = d_1|\cdots | d_k$ (with respect to $v$) to be those blocks $d_i$ of $u$ that are not also blocks of $v$. For example, with $u=31|2|4|6|7|85$ and $v = 76|85431|2$, the merging blocks of $u$ are $d_1 = \{1,3\}$, $d_3 = \{4\}$, $d_4=\{6\}$, $d_5 = \{7\}$, and $d_6 = \{5,8\}$. Further, by ``merging pairs" we mean those (unordered) pairs of merging blocks whose union is contained in a block of $v$. In the example above, the merging pairs are $(d_1,d_3)$, $(d_1, d_6)$, $(d_3,d_6)$, and $(d_4,d_5)$. We identify the ``position" of a merging pair with the position of its rightmost block, so that we say the merging pairs above occur in positions 3, 6, 6, and 5. 

If $u \prec w$ is a cover relation in $(S_n,\leq)$, then by rank considerations, there is only one pair of merging blocks in $u$, say $(d_i,d_j)$, whose union is a block in $w$. Define the label of this edge by its position: \[ \lambda_A(u,w) = \max\{i, j\}.\] For example if $u = 31|2|4|6|7|85$ and $w = 31|2|4|76|85$, $\lambda_A(u,w) = 5$, since the blocks that merged were the fourth and the fifth. See Figure \ref{fig:EL8} for more examples. Our goal is to prove the following.

\begin{prp}\label{prp:ELA}
\emph{(\cite[Theorem 4.3]{B}).} 
The labeling $\lambda_A$ is an EL-labeling for $(S_n,\leq)$.
\end{prp}

\begin{figure}
\begin{tikzpicture}
\tikzstyle{state1}=[rectangle,draw=black,scale=.55]
\draw (0,-1) node[state1] (u)
 {\begin{tikzpicture}[>=stealth',bend angle=45,auto,scale=.6]
 \tikzstyle{state}=[circle,draw=black,minimum size=4mm]
 \draw (1,3) node[state] (1) {$3$};
 \draw (2,1) node[state] (2) {$1$};
 \draw (3,2) node[state] (3) {$2$};
 \draw (4,4) node[state] (4) {$4$};
 \draw (5,6) node[state] (5) {$6$};
 \draw (6,7) node[state] (6) {$7$};
 \draw (7,8) node[state] (7) {$8$};
 \draw (8,5) node[state] (8) {$5$};
 \draw[line width=4] (1)--(2);
 \draw[line width=4] (7)--(8);
 \draw[dashed, ->] (1.5,2)--(3);
 \draw[dashed, ->] (5)--(7.5,6);
 \draw[dashed, ->] (6)--(7.3,7);
 \end{tikzpicture}};
 \draw (-6,4) node[state1] (u1)
 {\begin{tikzpicture}[>=stealth',bend angle=45,auto,scale=.6]
 \tikzstyle{state}=[circle,draw=black,minimum size=4mm]
 \draw (1,4) node[state] (1) {$4$};
 \draw (2,3) node[state] (2) {$3$};
 \draw (3,1) node[state] (3) {$1$};
 \draw (4,2) node[state] (4) {$2$};
 \draw (5,6) node[state] (5) {$6$};
 \draw (6,7) node[state] (6) {$7$};
 \draw (7,8) node[state] (7) {$8$};
 \draw (8,5) node[state] (8) {$5$};
 \draw[line width=4] (1)--(2);
 \draw[line width=4] (2)--(3);
 \draw[line width=4] (7)--(8);
 \draw[dashed, ->] (2.5,2)--(4);
 \draw[dashed, ->] (5)--(7.5,6);
 \draw[dashed, ->] (6)--(7.3,7);
 \end{tikzpicture}};
 \draw (-2,4) node[state1] (u2)
 {\begin{tikzpicture}[>=stealth',bend angle=45,auto,scale=.6]
 \tikzstyle{state}=[circle,draw=black,minimum size=4mm]
 \draw (1,4) node[state] (1) {$4$};
 \draw (2,6) node[state] (2) {$6$};
 \draw (3,7) node[state] (3) {$7$};
 \draw (4,8) node[state] (4) {$8$};
 \draw (5,5) node[state] (5) {$5$};
 \draw (6,3) node[state] (6) {$3$};
 \draw (7,1) node[state] (7) {$1$};
 \draw (8,2) node[state] (8) {$2$};
 \draw[line width=4] (4)--(5);
 \draw[line width=4] (5)--(6);
 \draw[line width=4] (6)--(7);
 \draw[dashed, ->] (1)--(5.3,4);
 \draw[dashed, ->] (6.5,2)--(8);
 \draw[dashed, ->] (2)--(4.5,6);
 \draw[dashed, ->] (3)--(4.3,7);
 \end{tikzpicture}};
 \draw (2,4) node[state1] (u3)
 {\begin{tikzpicture}[>=stealth',bend angle=45,auto,scale=.6]
 \tikzstyle{state}=[circle,draw=black,minimum size=4mm]
  \draw (1,6) node[state] (1) {$6$};
 \draw (2,7) node[state] (2) {$7$};
 \draw (3,8) node[state] (3) {$8$};
 \draw (4,5) node[state] (4) {$5$};
 \draw (5,3) node[state] (5) {$3$};
 \draw (6,1) node[state] (6) {$1$};
 \draw (7,2) node[state] (7) {$2$};
 \draw (8,4) node[state] (8) {$4$};
 \draw[line width=4] (3)--(4);
 \draw[line width=4] (4)--(5);
 \draw[line width=4] (5)--(6);
 \draw[dashed, ->] (4.5,4)--(8);
 \draw[dashed, ->] (5.5,2)--(7);
 \draw[dashed, ->] (1)--(3.5,6);
 \draw[dashed, ->] (2)--(3.3,7);
 \end{tikzpicture}};
 \draw (6,4) node[state1] (u4)
 {\begin{tikzpicture}[>=stealth',bend angle=45,auto,scale=.6]
 \tikzstyle{state}=[circle,draw=black,minimum size=4mm]
 \draw (1,3) node[state] (1) {$3$};
 \draw (2,1) node[state] (2) {$1$};
 \draw (3,2) node[state] (3) {$2$};
 \draw (4,6) node[state] (4) {$6$};
 \draw (5,7) node[state] (5) {$7$};
 \draw (6,8) node[state] (6) {$8$};
 \draw (7,5) node[state] (7) {$5$};
 \draw (8,4) node[state] (8) {$4$};
 \draw[line width=4] (1)--(2);
 \draw[line width=4] (6)--(7);
 \draw[line width=4] (7)--(8);
 \draw[dashed, ->] (1.5,2)--(3);
 \draw[dashed, ->] (4)--(6.5,6);
 \draw[dashed, ->] (5)--(6.3,7);
 \end{tikzpicture}};
 \draw (0,9) node[state1] (v)
 {\begin{tikzpicture}[>=stealth',bend angle=45,auto,scale=.6]
 \tikzstyle{state}=[circle,draw=black,minimum size=4mm]
 \draw (1,6) node[state] (1) {$6$};
 \draw (2,7) node[state] (2) {$7$};
 \draw (3,8) node[state] (3) {$8$};
 \draw (4,5) node[state] (4) {$5$};
 \draw (5,4) node[state] (5) {$4$};
 \draw (6,3) node[state] (6) {$3$};
 \draw (7,1) node[state] (7) {$1$};
 \draw (8,2) node[state] (8) {$2$};
 \draw[line width=4] (3)--(4);
 \draw[line width=4] (4)--(5);
 \draw[line width=4] (5)--(6);
 \draw[line width=4] (6)--(7);
 \draw[dashed, ->] (1)--(3.5,6);
 \draw[dashed, ->] (2)--(3.35,7);
 \draw[dashed, ->] (6.5,2)--(8);
 \end{tikzpicture}};
\draw[thick] (u) -- node[auto=left] {3} (u1);
\draw[thick] (u) -- node[auto=left] {6} (u2);
\draw[thick] (u) -- node[auto=right] {6} (u3);
\draw[thick] (u) -- node[auto=right] {6} (u4);
\draw[thick] (v) -- node[auto=right] {5} (u1);
\draw[thick] (v) -- node[auto=right] {4} (u2);
\draw[thick] (v) -- node[auto=left] {5} (u3);
\draw[thick] (v) -- node[auto=left] {5} (u4);
\end{tikzpicture}
\caption{The edge labelings for an interval of length two with three merging blocks.}\label{fig:EL8}
\end{figure}
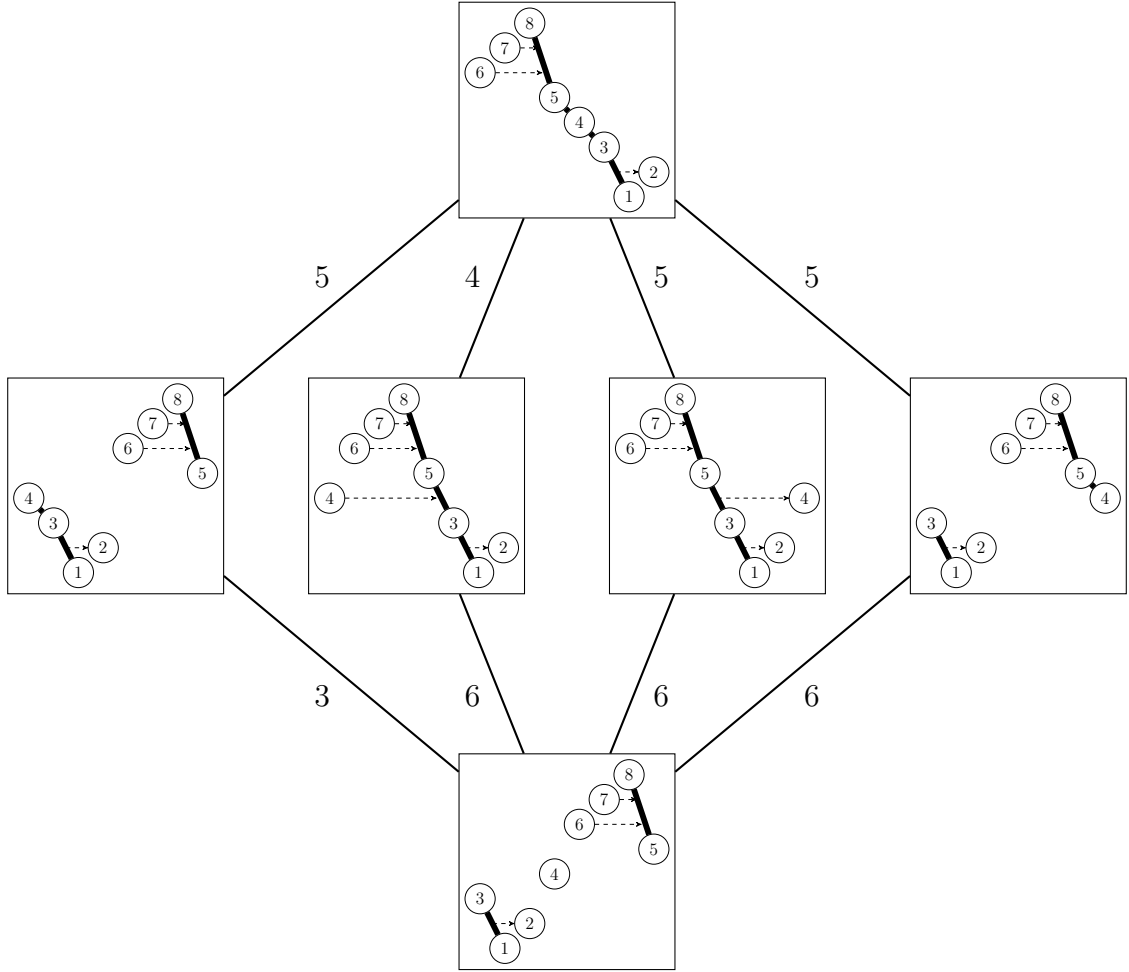

Notice that merging (with respect to a fixed $v>u$) defines a transitive relation on the blocks of $u$, i.e., if $(d_i,d_j)$ is a merging pair and $(d_j,d_k)$ is a merging pair, then $(d_i,d_k)$ is also a merging pair. This follows since the first two pairings imply that $d_i \cup d_j \cup d_k$ is contained in a block of $v$. 

This transitivity implies that, for a given $v>u$, the leftmost merging pair of $u$ is uniquely defined. That is, suppose $(d_i,d_j)$ is a merging block with position $j>i$ that has $j$ minimal among all merging pairs of $u$, and suppose $(d_k,d_j)$ is another merging pair in position $j$. Then either $i=k$, or transitivity of merging pairs implies $(d_i,d_k)$ is a merging pair, found in position $\max\{i,k\} < j$, contradicting the minimality of position $j$. With the labeling $\lambda_A$ in mind, this observation can be phrased as follows. 

\begin{obs}\label{obs:leftmost} 
Let $u \prec w \leq v$ be a triple of elements in $(S_n,\leq)$ for which $w$ is the cover of $u$ obtained by merging the leftmost pair of blocks (with respect to $v$). Then if $w' \neq w$ is any other cover of $u$ with $u \prec w' \leq v$, \[ \lambda_A(u,w) < \lambda_A(u,w').\]
\end{obs}

Hence, every interval $[u,v]$ has a unique lexicographically first maximal chain obtained by merging pairs of blocks greedily from left to right. This chain: \begin{equation}\label{eq:c} c: u=u_0 \prec u_1 \prec \cdots \prec u_r =v\end{equation} has $u_i$ obtained from $u_{i-1}$ by merging the leftmost pair of blocks in $u_{i-1}$ (with respect to $v$).

To prove $\lambda_A$ is an EL-labeling, we still need to show that $c$ is a rising chain and that every other chain has a fall. To this end, it will be helpful to understand how the positions of merging blocks can change in moving up a chain in the interval $[u,v]$. 

Suppose $u = d_1|\cdots |d_k$, $w = e_1|\cdots | e_{k-1}$ and $u\prec w \leq v$. Suppose that $d_i$ and $d_j$ are the blocks of $u$ that merge in $w$. Then $d_i \cup d_j = e_{j'}$ for some $i\leq j' \leq j-1$ while the other blocks of $w$ are the other blocks of $u$, and they appear in more or less the same positions. Specifically,
\[
e_a = \begin{cases} d_a &\mbox{if } a < i \\
				   d_{a+1} &\mbox{if } a \geq j.
				   \end{cases}
\]
That is, a cover $w$ of $u$ typically looks like: 
\begin{align*} w &= e_1|\cdots|e_{i-1}|e_i| \cdots |e_{j-1}|e_j| \cdots |e_{k-1} \\
&= d_1| \cdots | d_{i-1}| e_i | \cdots | e_{j-1}| d_{j+1}|\cdots | d_k.
\end{align*}
Further, only one of the blocks among $e_i, \ldots, e_{j-1}$, i.e., $e_{j'} = d_i \cup d_j$, is not also block of $u$: \[ \{ e_i, \ldots, e_{j-1} \} = \{ d_{i+1}, \ldots, d_{j-1}, e_{j'}\}.\]

In particular, if $w=u_1$ is the cover of $u$ obtained by merging the leftmost merging pair, then the leftmost merging pair of $w$ is in a position greater than or equal to $j$ (that is, its righmost block is $d_{a+1}$ for some $a \geq j$). Applying this logic to every link in the lexicographically first chain $c$ from \eqref{eq:c}, we see the lexicographically first chain $c$ is a rising chain: \[ \lambda_A(u,u_1) \leq \lambda_A(u_1,u_2) \leq \cdots \leq \lambda_A(u_{r-1},v).\]

On the other hand, if $w \neq u_1$, the same characterization shows the leftmost merging pair of $w$ is in a position $b\leq j-1$. We can now use induction on the length of the interval to conclude that any chain other than $c$ has a fall. That is, suppose for induction that the only rising maximal chain in the interval $[w,v]$ is its lexicographically first chain, $c': w=w_0 \prec w_1 \prec \cdots \prec w_{r-1}=v$. Prepending $u$ to this chain gives a fall, since $\lambda_A(u,w)=j > b=\lambda_A(w,w_1)$. Any chain formed by prepending $u$ to another saturated chain from $w$ to $v$ already has a fall somewhere between $w$ and $v$ by the induction hypothesis.

This completes the proof that $\lambda_A$ is an EL-labeling, as stated in Proposition \ref{prp:ELA}.

\subsection{Proof of shellability for type $B_n$}

The labeling for type $B_n$ generalizes the idea for type $A_{n-1}$. The elements of $B_n$ are written in long form, as symmetric signed permutations $u = d_{-k}|\cdots |d_{-1}|d_0|d_1|\cdots |d_k$, where the blocks are again decreasing runs, and $d_{-j} = \{ -m : m \in d_j\}$. We can still apply the language of ``merging blocks" and ``merging pairs" as we did in type $A_{n-1}$. The primary difference here is that when we merge blocks $d_i$ and $d_j$, the blocks $d_{-i}$ and $d_{-j}$ also merge. Therefore when referring to a merging pair we really have two symmetric pairs of merging blocks. In the special case of $i=-j$, three blocks merge to form $d_{-j} \cup d_0 \cup d_j$. (This is also equivalent to the case $i=0$.) Nonetheless, we identify the position of the pair with the righmost block among $d_i, d_{-i}, d_j, d_{-j}$, so the position of the merging pair $(d_i,d_j)$ is $\max\{ |i|, |j|\}$. For simplicity, we will identify both merging pairs by the pair $(d_i,d_j)$ with $0\leq |i| < j$.

For example if $u = \bar 4 \bar 5| \bar 3 |\bar 2 | 1 0 \bar 1| 2 | 3| 54$ and $v = 5 4 \bar2 \bar3 | 1 0\bar1 | 3 2\bar 4\bar 5$, the merging blocks of $u$ (with respect to $v$) are: $d_{-3} = \{\bar 5, \bar 4\}, d_{-2} = \{\bar 3\}, d_{-1}=\{\bar 2\}, d_1 = \{2\}, d_2=\{3\}, d_3=\{4,5\}$. The merging pairs are: $(d_{-2},d_3)$ (with symmetric pair $(d_{-3},d_2)$), $(d_{-1},d_3)$  (with $(d_{-3},d_1)$), and $(d_1,d_2)$ (with $(d_{-2},d_{-1})$), which we say are in positions 3, 3, and 2, respectively.

We define the edge labeling for $(B_n,\leq)$ in terms of the position of the merging blocks just as in type $A_{n-1}$. If $w$ covers $u$ in $(B_n, \leq)$ and $(d_i,d_j)$ and $(d_{-i},d_{-j})$ are the merging blocks of $u$, \[ \lambda_B(u,w) = \max\{|i|,|j|\}.\] For example if $u = \bar 4 \bar 5| \bar 3 |\bar 2 | 1 0 \bar 1| 2 | 3| 54$ and $w= \bar3|54\bar2|10\bar1|2\bar4 \bar5|3$, then $\lambda_B(u,w) = 3$. We will prove the following analogue of Proposition \ref{prp:ELA}.

\begin{prp}\label{prp:ELB}
The labeling $\lambda_B$ is an EL-labeling for $(B_n,\leq)$.
\end{prp}

Notice that this subsumes Proposition \ref{prp:ELA}, since $(S_n,\leq)$ is an interval in $(B_n,\leq)$, and if $u, w\in S_n$, $\lambda_B(u,w) = \lambda_A(u,w)$.

Because of the symmetric nature of type $B_n$ permutation pre-orders we will refer to ``innermost" rather than ``leftmost" merging pairs. Again, transitivity of the merging relation on blocks implies the innermost merging pair $(d_i,d_j)$, with $0\leq |i|<j$, is the only block in position $j$. Thus we have the following analogue of Observation \ref{obs:leftmost}.

\begin{obs}\label{obs:leftmostB}
Let $u \prec w \leq v$ be a triple of elements in $(B_n,\leq)$ for which $w$ is the cover of $u$ obtained by merging the innermost pair of blocks (with respect to $v$). Then if $w' \neq w$ is any other cover of $u$ with $u \prec  w' \leq v$, \[ \lambda_B(u,w) < \lambda_B(u,w').\]
\end{obs}

As in the type $A_{n-1}$ case, this observation immediately implies that in any interval $[u,v]$ there is a uniquely defined lexicographcially first maximal chain. This chain: \[ c: u = u_0 \prec u_1 \prec  \cdots \prec  u_r = v \] is again formed by merging the innermost blocks at each cover. We will now see that $c$ is a rising chain in $[u,v]$, and that every other maximal chain in $[u,v]$ has a fall.

Let $u = d_{-k}|\cdots |d_{-1}|d_0|d_1|\cdots |d_k$ and let $w = e_{-k+1}|\cdots|e_{-1}|e_0|e_1|\cdots |e_{k-1}$ be a cover of $u$ such that $u\prec w\leq v$. Let $0\leq |i|< j$ and suppose $(d_i,d_j)$ and $(d_{-i},d_{-j})$ are the blocks of $u$ that merge to give $w$. That is, $d_i \cup d_j = e_{j'}$ and $d_{-i}\cup d_{-j} = e_{-j'}$, where $j'$ is such that $-j+1\leq j' \leq j-1$. Then the blocks farther from the center than $d_j$ remain untouched. If $i>0$ as well, the blocks closer to the center than $i$ also remain in place. We have:
\[
e_a = \begin{cases}  d_a &\mbox{if } |a| < i,\\
					d_{a+1} &\mbox{if } a \geq j, \\
				   d_{a-1} &\mbox{if } a \leq -j.				 
				   \end{cases}
\]
That is, if $i > 0$, a cover $w$ of $u$ typically looks like: 
\begin{align*} w &= e_{-k+1}|\cdots|e_{-j}|e_{-j+1}|\cdots|e_{-i}|e_{-i+1}|\cdots| e_{-1}|e_0|e_1|\cdots|e_{i-1}|e_i| \cdots |e_{j-1}|e_{j}| \cdots |e_{k-1} \\
&= d_{-k}|\cdots|d_{-j-1}|e_{-j+1}|\cdots|e_{-i}|d_{-i+1}|\cdots| d_{-1}|d_0|d_1|\cdots|d_{i-1}|e_i| \cdots |e_{j-1}|d_{j+1}| \cdots |d_{k}.
\end{align*}
To emphasize the difference with type $A_{n-1}$, if $i\leq 0$, a cover $w$ of $u$ is:
\begin{align*} w &= e_{-k+1}|\cdots|e_{-j}|e_{-j+1}|\cdots| e_{-1}|e_0|e_1|\cdots|e_{j-1}|e_{j}| \cdots |e_{k-1} \\
&= d_{-k}|\cdots|d_{-j-1}|e_{-j+1}|\cdots| e_{-1}|e_0|e_1|\cdots |e_{j-1}|d_{j+1}| \cdots |d_{k}.
\end{align*}
In either situation, nearly all the blocks between $d_{-j-1}$ and $d_{j+1}$ are also blocks of $u$: \[ \{ e_{-j+1}, \ldots, e_{j-1} \} = \begin{cases} \{ d_{-j+1}, \ldots, d_{j-1}\} - \{ d_i, d_{-i}, d_j, d_{-j} \} \cup \{ d_i \cup d_j, d_{-i} \cup d_{-j} \} & \mbox{ if $0 < |i|<j$},\\
\{ d_{-j+1}, \ldots, d_{j-1}\} - \{ d_0, d_j, d_{-j} \} \cup \{ d_{-j} \cup d_0 \cup d_{j} \} & \mbox{ if $i=0$}.
\end{cases}
\]

In particular, if $w=u_1$ is the cover of $u$ obtained by merging the innermost merging pair, then the innermost merging pair of $u_1$ is in a position $a \geq j$. Of course the same reasoning applies to every cover $u_{i-1} \prec  u_i$, so this shows that the chain $c$ is a rising chain:
\[ \lambda_B(u,u_1) \leq \lambda_B(u_1,u_2) \leq \cdots \leq \lambda_B(u_{r-1},v).\]

To see that any other maximal chain in $[u,v]$ has a fall, we argue by induction on the length of the interval. Suppose $w\neq u_1$ is a cover of $u$. By induction, we suppose the lexicographically first chain, $c': w=w_0 \prec  w_1 \prec  \cdots \prec  w_{r-1} = v$, is the only rising chain in $[w,v]$. However, the innermost merging pair of $w$ is in a position $b \leq j-1$, so $\lambda_B(u,w) = j > b = \lambda_B(w,w_1)$, and thus every saturated chain $u \prec w \prec \cdots \prec v$ must have a fall.

This completes the proof of Proposition \ref{prp:ELB}.

\subsection{Proof of shellability for type $D_n$}

The main idea for type $D_n$ is much the same as for type $B_n$. For $u \in D_n$ we write a forked signed permutation with bars in ascent positions: \[ u = d_{-k}|\cdots |d_{-1}|d_0|d_1|\cdots |d_k \quad \mbox{ or }\quad  u = d_{-k}|\cdots |d_{-2}|d_{-1},d_1|d_2|\cdots |d_k.\] The main difference from type $B_n$ is that the ``zero block" can consist of one block, $d_0$, or the two incomparable blocks $d_1$ and $d_{-1}$. In this case, the block $d_1$ is distinguished from $d_{-1}$ in that $|u(1)| \in d_1$. For example, in $\bar 2 | 31 | 5\begin{array}{c} 4\\ \bar 4\end{array} \bar 5| \bar 1 \bar 3| 2$, we have $d_0 = \{ \bar 5, \bar 4, 4, 5\}$, while in $\bar 2 | 31 | 4\begin{array}{ c} 5 \\ \bar 5 \end{array} \bar 4| \bar 1 \bar 3| 2$ we have $d_{-1} = \{\bar 5, 4\}$ and $d_1 = \{ \bar 4, 5\}$. 

We will define the merging blocks, merging pairs, and so on in the same way as for type $B_n$. Roughly speaking our EL-labeling will work the same way, with edges labeled by positions of merging pairs, though we have a different way to define position of a merging pair. The unique rising chain is lexicographically first, and it is obtained greedily, by merging the pairs in the smallest possible position at each rank.

Before proceeding, let us address the new notion of the ``position" of a pair of blocks in a type $D_n$ shard intersection. Let $u = d_{-k}|\cdots |d_k \in D_n$. If $0\leq i < j$, we say the pair $(d_i,d_j)$ (and its symmetric partner $(d_{-i},d_{-j})$) are in position $j$, as before. We call such a merge a ``positive" merge. However, if $0<-i < j$, that is, if $d_i$ and $d_j$ are on opposite sides of the center of $u$, then we call this type of merge a ``negative" merge. In this case we define $(d_i,d_j)$ (and $(d_{-i},d_{-j})$) to be in position $j+|u|$, where $|u|$ is the co-rank of $u$ in $(D_n,\leq)$. Note that $|u|=k$ is also the largest index on a block of $u$ and, by Observation \ref{obs:desD}, the dimension of the cone $C(u)$.

We define the labeling $\lambda_D$ as follows. Suppose $u \prec w$ is a cover in $(D_n,\leq)$ and that $(d_i,d_j)$ $(d_{-i},d_{-j})$ are the blocks of $u$ that are merged in $w$. We label this cover with the position of the merged pair. Suppose, without loss of generality, that $0\leq |i| < j$. Then,
\[
\lambda_D(u,w) = \begin{cases} j & \mbox{if } 0 \leq i < j, \\
 j + |u| & \mbox{if } 0 < -i < j.
\end{cases}
\]

For example, if $u = \bar 2 | 31 | 4\begin{array}{ c} 5 \\ \bar 5 \end{array} \bar 4| \bar 1 \bar 3| 2$, we could merge blocks $d_1$ and $d_3$ (and $d_{-1}$ with $d_{-3}$) to get $w = 31 | 4 \bar 2 \begin{array}{ c} 5 \\ \bar 5 \end{array} 2 \bar 4| \bar 1 \bar 3$. This is cover comes from a positive merge, and so would get the label $\lambda_D(u,w) = 3$.

As a different example, take $u = \bar 2 | 31 | 4\begin{array}{ c} 5 \\ \bar 5 \end{array} \bar 4| \bar 1 \bar 3| 2$ again, but merge blocks $d_{-1}$ and $d_2$ to get $w = \bar 2 | 531  \begin{array}{c} 4\\ \bar 4\end{array} \bar 1 \bar 3 \bar 5|2$. As this cover comes from a negative merge it is labeled as $\lambda_D(u,w) = j + |u| = 2 + 3 =5$. More examples of the edge labeling can be seen in Figure \ref{fig:ELD}.

\begin{figure}
\begin{tikzpicture}[scale=.75]
\tikzstyle{state1}=[rectangle,draw=black,scale=.4]
\draw (0,-1) node[state1] (u)
 {\begin{tikzpicture}[>=stealth',bend angle=45,auto,scale=.6]
 \tikzstyle{state}=[circle,draw=black,minimum size=4mm]
 \draw (-3,-4) node[state] (1) {$\bar 4$};
 \draw (-2,-3) node[state] (2) {$\bar 3$};
 \draw (-1,-2) node[state] (3) {$\bar 2$};
 \draw (0,-1) node[state] (4) {$\bar 1$};
 \draw (0,1) node[state] (5) {$1$};
 \draw (1,2) node[state] (6) {$2$};
 \draw (2,3) node[state] (7) {$3$};
 \draw (3,4) node[state] (8) {$4$};
 \end{tikzpicture}};
 \draw (-6,4) node[state1] (u1)
 {\begin{tikzpicture}[>=stealth',bend angle=45,auto,scale=.6]
 \tikzstyle{state}=[circle,draw=black,minimum size=4mm]
 \draw (-3,-4) node[state] (1) {$\bar 4$};
 \draw (-2,-3) node[state] (2) {$\bar 3$};
 \draw (-1,-1) node[state] (3) {$\bar 1$};
 \draw (0,-2) node[state] (4) {$\bar 2$};
 \draw (0,2) node[state] (5) {$2$};
 \draw (1,1) node[state] (6) {$1$};
 \draw (2,3) node[state] (7) {$3$};
 \draw (3,4) node[state] (8) {$4$};
 \draw[line width=4] (3)--(4);
 \draw[line width=4] (5)--(6);
 \end{tikzpicture}};
 \draw (6,4) node[state1] (v1)
 {\begin{tikzpicture}[>=stealth',bend angle=45,auto,scale=.6]
 \tikzstyle{state}=[circle,draw=black,minimum size=4mm]
 \draw (-3,-4) node[state] (1) {$\bar 4$};
 \draw (-2,-3) node[state] (2) {$\bar 3$};
 \draw (-1,1) node[state] (3) {$1$};
 \draw (0,-2) node[state] (4) {$\bar 2$};
 \draw (0,2) node[state] (5) {$2$};
 \draw (1,-1) node[state] (6) {$\bar 1$};
 \draw (2,3) node[state] (7) {$3$};
 \draw (3,4) node[state] (8) {$4$};
 \draw[line width=4] (3)--(4);
 \draw[line width=4] (5)--(6);
 \end{tikzpicture}};
 \draw (0,4) node[state1] (w1)
 {\begin{tikzpicture}[>=stealth',bend angle=45,auto,scale=.6]
 \tikzstyle{state}=[circle,draw=black,minimum size=4mm]
 \draw (-3,-3) node[state] (1) {$\bar 3$};
 \draw (-2,-2) node[state] (2) {$\bar 2$};
 \draw (-1,-4) node[state] (3) {$\bar 4$};
 \draw (0,-1) node[state] (4) {$\bar 1$};
 \draw (0,1) node[state] (5) {$1$};
 \draw (1,4) node[state] (6) {$4$};
 \draw (2,2) node[state] (7) {$2$};
 \draw (3,3) node[state] (8) {$3$};
 \draw[line width=4] (2)--(3);
 \draw[line width=4] (6)--(7);
 \draw[dashed,->] (1)--(-1.65,-3);
 \draw[dashed,->] (1.5,3)--(8);
 \end{tikzpicture}};
 \draw (-10,9) node[state1] (u2)
 {\begin{tikzpicture}[>=stealth',bend angle=45,auto,scale=.6]
 \tikzstyle{state}=[circle,draw=black,minimum size=4mm]
 \draw (-3,-4) node[state] (1) {$\bar 4$};
 \draw (-2,-1) node[state] (2) {$\bar 1$};
 \draw (-1,-2) node[state] (3) {$\bar 2$};
 \draw (0,-3) node[state] (4) {$\bar 3$};
 \draw (0,3) node[state] (5) {$3$};
 \draw (1,2) node[state] (6) {$2$};
 \draw (2,1) node[state] (7) {$1$};
 \draw (3,4) node[state] (8) {$4$};
 \draw[line width=4] (2)--(3)--(4);
 \draw[line width=4] (5)--(6)--(7);
 \end{tikzpicture}};
 \draw (0,9) node[state1] (v2)
 {\begin{tikzpicture}[>=stealth',bend angle=45,auto,scale=.6]
 \tikzstyle{state}=[circle,draw=black,minimum size=4mm]
 \draw (-3,-4) node[state] (1) {$\bar 4$};
 \draw (-2,-3) node[state] (2) {$\bar 3$};
 \draw (-1,2) node[state] (3) {$2$};
 \draw (0,-1) node[state] (4) {$\bar 1$};
 \draw (0,1) node[state] (5) {$1$};
 \draw (1,-2) node[state] (6) {$\bar 2$};
 \draw (2,3) node[state] (7) {$3$};
 \draw (3,4) node[state] (8) {$4$};
 \draw[line width=4] (3)--(4)--(6);
 \draw[line width=4] (3)--(5)--(6);
 \end{tikzpicture}};
 \draw (-5,9) node[state1] (w2)
 {\begin{tikzpicture}[>=stealth',bend angle=45,auto,scale=.6]
 \tikzstyle{state}=[circle,draw=black,minimum size=4mm]
 \draw (-3,-3) node[state] (1) {$\bar 3$};
 \draw (-2,-1) node[state] (2) {$\bar 1$};
 \draw (-1,-2) node[state] (3) {$\bar 4$};
 \draw (0,-4) node[state] (4) {$\bar 2$};
 \draw (0,4) node[state] (5) {$4$};
 \draw (1,2) node[state] (6) {$2$};
 \draw (2,1) node[state] (7) {$1$};
 \draw (3,3) node[state] (8) {$3$};
 \draw[line width=4] (2)--(3)--(4);
 \draw[line width=4] (5)--(6)--(7);
 \draw[dashed,->] (1)--(-.65,-3);
 \draw[dashed,->] (.65,3)--(8);
 \end{tikzpicture}};
 \draw (10,9) node[state1] (x2)
 {\begin{tikzpicture}[>=stealth',bend angle=45,auto,scale=.6]
 \tikzstyle{state}=[circle,draw=black,minimum size=4mm]
 \draw (-3,-4) node[state] (1) {$\bar 4$};
 \draw (-2,1) node[state] (2) {$1$};
 \draw (-1,-2) node[state] (3) {$\bar 2$};
 \draw (0,-3) node[state] (4) {$\bar 3$};
 \draw (0,3) node[state] (5) {$3$};
 \draw (1,2) node[state] (6) {$2$};
 \draw (2,-1) node[state] (7) {$\bar 1$};
 \draw (3,4) node[state] (8) {$4$};
 \draw[line width=4] (2)--(3)--(4);
 \draw[line width=4] (5)--(6)--(7);
 \end{tikzpicture}};
 \draw (5,9) node[state1] (y2)
 {\begin{tikzpicture}[>=stealth',bend angle=45,auto,scale=.6]
 \tikzstyle{state}=[circle,draw=black,minimum size=4mm]
 \draw (-3,-3) node[state] (1) {$\bar 3$};
 \draw (-2,1) node[state] (2) {$1$};
 \draw (-1,-2) node[state] (3) {$\bar 2$};
 \draw (0,-4) node[state] (4) {$\bar 4$};
 \draw (0,4) node[state] (5) {$4$};
 \draw (1,2) node[state] (6) {$2$};
 \draw (2,-1) node[state] (7) {$\bar 1$};
 \draw (3,3) node[state] (8) {$3$};
 \draw[line width=4] (2)--(3)--(4);
 \draw[line width=4] (5)--(6)--(7);
 \draw[dashed,->] (1)--(-.65,-3);
 \draw[dashed,->] (.65,3) -- (8);
 \end{tikzpicture}};
 \draw (-6,14) node[state1] (u3)
 {\begin{tikzpicture}[>=stealth',bend angle=45,auto,scale=.6]
 \tikzstyle{state}=[circle,draw=black,minimum size=4mm]
 \draw (-3,-1) node[state] (1) {$\bar 1$};
 \draw (-2,-2) node[state] (2) {$\bar 2$};
 \draw (-1,-3) node[state] (3) {$\bar 3$};
 \draw (0,-4) node[state] (4) {$\bar 4$};
 \draw (0,4) node[state] (5) {$4$};
 \draw (1,3) node[state] (6) {$3$};
 \draw (2,2) node[state] (7) {$2$};
 \draw (3,1) node[state] (8) {$1$};
 \draw[line width=4] (1)--(2)--(3)--(4);
 \draw[line width=4] (5)--(6)--(7)--(8);
 \end{tikzpicture}};
 \draw (0,14) node[state1] (v3)
 {\begin{tikzpicture}[>=stealth',bend angle=45,auto,scale=.6]
 \tikzstyle{state}=[circle,draw=black,minimum size=4mm]
 \draw (-3,-4) node[state] (1) {$\bar 4$};
 \draw (-2,3) node[state] (2) {$3$};
 \draw (-1,2) node[state] (3) {$2$};
 \draw (0,-1) node[state] (4) {$\bar 1$};
 \draw (0,1) node[state] (5) {$1$};
 \draw (1,-2) node[state] (6) {$\bar 2$};
 \draw (2,-3) node[state] (7) {$\bar 3$};
 \draw (3,4) node[state] (8) {$4$};
 \draw[line width=4] (2)--(3)--(4)--(6);
 \draw[line width=4] (3)--(5)--(6)--(7);
 \end{tikzpicture}};
 \draw (6,14) node[state1] (x3)
 {\begin{tikzpicture}[>=stealth',bend angle=45,auto,scale=.6]
 \tikzstyle{state}=[circle,draw=black,minimum size=4mm]
 \draw (-3,1) node[state] (1) {$1$};
 \draw (-2,-2) node[state] (2) {$\bar 2$};
 \draw (-1,-3) node[state] (3) {$\bar 3$};
 \draw (0,-4) node[state] (4) {$\bar 4$};
 \draw (0,4) node[state] (5) {$4$};
 \draw (1,3) node[state] (6) {$3$};
 \draw (2,2) node[state] (7) {$2$};
 \draw (3,-1) node[state] (8) {$\bar 1$};
 \draw[line width=4] (1)--(2)--(3)--(4);
 \draw[line width=4] (5)--(6)--(7)--(8);
 \end{tikzpicture}};
 \draw (0,19) node[state1] (v)
 {\begin{tikzpicture}[>=stealth',bend angle=45,auto,scale=.6]
 \tikzstyle{state}=[circle,draw=black,minimum size=4mm]
 \draw (-3,4) node[state] (1) {$4$};
 \draw (-2,3) node[state] (2) {$3$};
 \draw (-1,2) node[state] (3) {$2$};
 \draw (0,1) node[state] (4) {$1$};
 \draw (0,-1) node[state] (5) {$\bar 1$};
 \draw (1,-2) node[state] (6) {$\bar 2$};
 \draw (2,-3) node[state] (7) {$\bar 3$};
 \draw (3,-4) node[state] (8) {$\bar 4$};
 \draw[line width=4] (1)--(2)--(3)--(4)--(6);
 \draw[line width=4] (3)--(5)--(6)--(7)--(8);
 \end{tikzpicture}};
\draw[thick] (u) -- node[auto=left] {2} (u1);
\draw[thick] (u) -- node[auto=left] {4} (w1);
\draw[thick] (u) -- node[auto=right] {6} (v1);
\draw[thick] (u1) -- node[auto=left] {2} (u2);
\draw[thick] (u1) -- node[pos=.25, auto=right] {4} (v2);
\draw[thick] (u1) -- node[auto=left] {3} (w2);
\draw[thick] (u2) -- node[auto=left] {2} (u3);
\draw[thick] (u2) -- node[pos=.75, auto=left] {3} (v3);
\draw[thick] (u3) -- node[auto=left] {2} (v);
\draw[thick] (v1) -- node[pos=.25, auto=left] {4} (v2);
\draw[thick] (v1) -- node[auto=right] {2} (x2);
\draw[thick] (v1) -- node[auto=right] {3} (y2);
\draw[thick] (v2) -- node[auto=right] {1} (v3);
\draw[thick] (v3) -- node[auto=right] {1} (v);
\draw[thick] (w1) -- node[pos=.25,auto=right] {2} (w2);
\draw[thick] (w1) -- node[pos=.25, auto=left] {5} (y2);
\draw[thick] (w2) -- node[pos=.75, auto=left] {2} (u3);
\draw[thick] (x2) -- node[pos=.75, auto=right] {3} (v3);
\draw[thick] (x2) -- node[auto=right] {2} (x3);
\draw[thick] (y2) -- node[pos=.75, auto=right] {2} (x3);
\draw[thick] (x3) -- node[auto=right] {2} (v);
\end{tikzpicture}
\caption{Edge labelings between some elements of $(D_4,\leq)$. (Many elements and edges are omitted.)}\label{fig:ELD}
\end{figure}
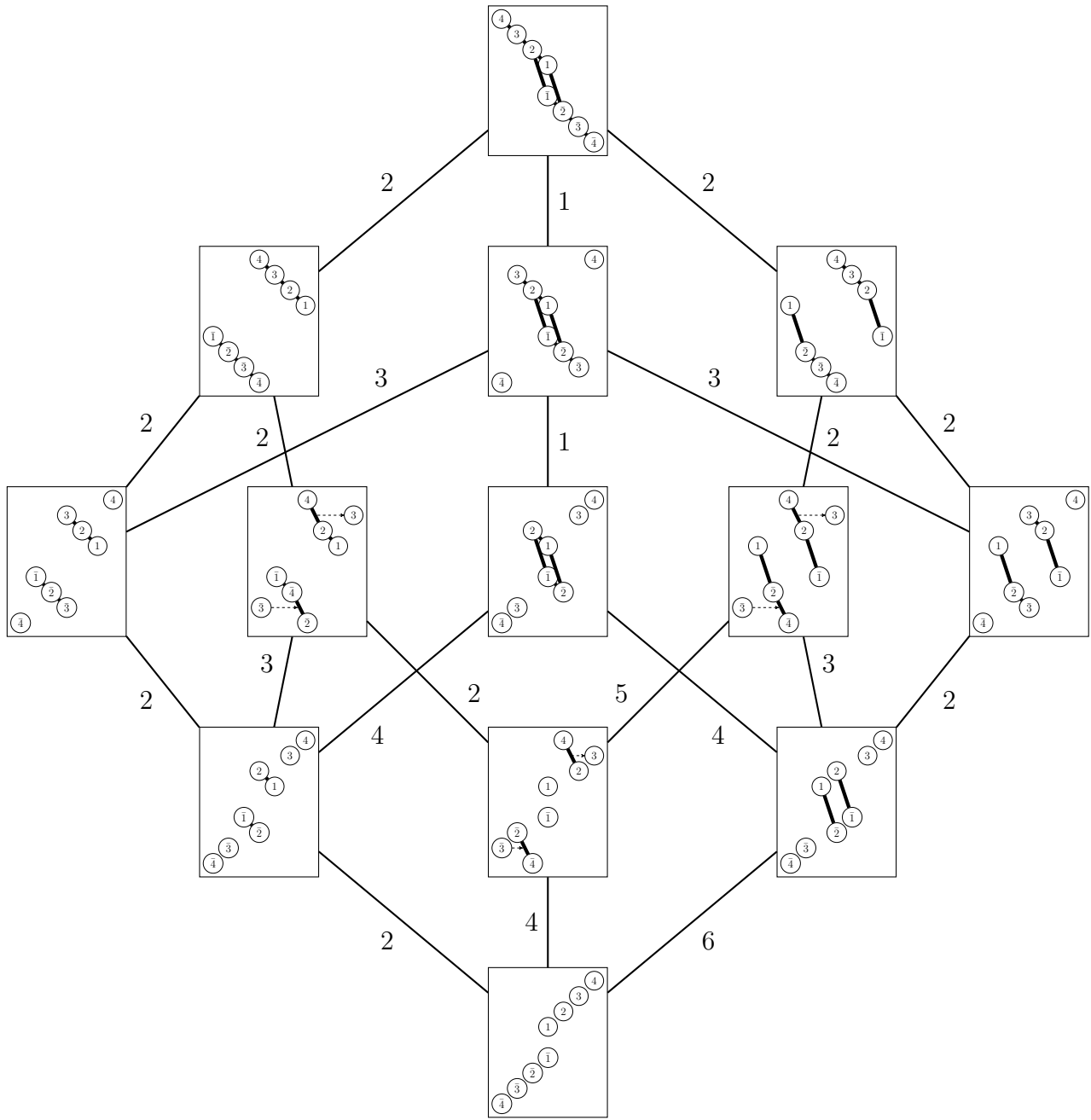

The need for this new sort of labeling stems from the fact that while the merging relation on blocks is still transitive, not every merging pair of $u$ (with respect to some $v>u$) can be merged to obtain a cover of $u$. To take a simple example, consider the interval from the identity to the long element in $(D_4,\leq)$: $u = \bar 4 | \bar 3 | \bar 2 | \begin{array}{c} 1 \\ \bar 1 \end{array} | 2 | 3 | 4$ and $v = 432\begin{array}{c} 1\\ \bar 1\end{array} \bar 2 \bar 3 \bar 4$. In $v$, the elements $1$ and $\bar 1$ are in the same block, so $d_{-1} = \{\bar 1\}$ and $d_1 = \{1\}$ are merging blocks of $u$. In the sense of type $B_n$ shard intersections, $(d_{-1},d_1)$ is the innermost pair. However, there is no shard of type $D_n$ with a block consisting solely of $1$ and $\bar 1$, so we cannot merge this pair at this time. Hence, we chose our labeling to be biased against merging positively indexed blocks with negatively indexed ones; what we are calling ``negative" merges. Note also that $(d_1,d_2)$ and $(d_{-1},d_2)$ are merging pairs of $u$. In the type $B_n$ setting, both pairs are in position 2, but in type $D_n$, the pair $(d_1,d_2)$ is in position 2, while $(d_{-1},d_2)$ is in position $2+|u| > 2$.

Roughly speaking, the lexicographically first chain merges positive pairs first, from inside out, then negative pairs, again from inside out. For example, the following is the lexicographically first maximal chain in $(D_4,\leq)$: 
\[ c: \bar 4 | \bar 3 | \bar 2 | \begin{array}{c} 1 \\ \bar 1 \end{array} | 2 | 3 | 4 \xrightarrow{2} \bar 4 | \bar 3 | \bar 1  \begin{array}{c} 2 \\ \bar 2 \end{array}  1 | 3 | 4 \xrightarrow{2} \bar 4 | \bar 1  \bar 2  \begin{array}{c} 3 \\ \bar 3 \end{array}  2  1 | 4 \xrightarrow{2} \bar 1  \bar 2  \bar 3  \begin{array}{c} 4 \\ \bar 4 \end{array}  3  2  1 \xrightarrow{2}  432\begin{array}{c} 1\\ \bar 1\end{array} \bar 2 \bar 3 \bar 4,\] where we have labeled all but the last cover with the position of the rightmost merging block. The final 2 comes from merging blocks $d_{-1}$ and $d_1$ in an element of co-rank 1, yielding the label $1+1=2$. See Figure \ref{fig:ELD} to find more examples of labeled chains of type $D_n$.

\begin{prp}\label{prp:ELD}
The labeling $\lambda_D$ is an EL-labeling for $(D_n,\leq)$.
\end{prp}

To show the lexicographically first chain is well-defined, we need to show that for any $u < v$ in $D_n$, there is only one merging pair of $u$ (with respect to $v$) in minimal position. Suppose $(d_i,d_j)$, with $0\leq |i| < j$, is a merging pair of $u$ in minimal position. If $i \geq 0$, the uniqueness of the pair follows from transitivity as in type $B_n$. Now suppose $i < 0$, so this minimal position is $j + |u|$. Suppose $(d_k,d_j)$, is another pair in position $j+|u|$. Then $0 < -k < j$, and if $k \neq i$, transitivity of merging implies that $(d_i,d_k)$ is a merging pair in $u$ as well. But since $i$ and $k$ are both negative indices, this pair is in position $\max\{-i,-k\} < j < j+|u|$, a contradiction. This yields the following analogue of Observations \ref{obs:leftmost} and \ref{obs:leftmostB}.

\begin{obs}\label{obs:leftmostD}
Let $u \prec w \leq v$ be a triple of elements in $(D_n,\leq)$ for which $w$ is the cover of $u$ obtained by merging the pair of blocks in minimal position (with respect to $v$). Then if $w' \neq w$ is any other cover of $u$ with $u \prec  w' \leq v$, \[ \lambda_D(u,w) < \lambda_D(u,w').\]
\end{obs}

Observation \ref{obs:leftmostD} implies there is a unique lexicographically first maximal chain in any interval $[u,v]$ of $(D_n,\leq)$. Call this chain: \[ c: u=u_0 \prec u_1 \prec  \cdots \prec  u_r = v. \] To see that this is the only rising chain, we make the same simple observation about the positions of blocks in a cover $u\prec w$. Namely, if $(d_i,d_j)$ is a merging pair with $0\leq |i| < j$, then blocks farther from center than $d_j$ stay in place. Letting $u = d_{-k}|\cdots |d_k$ and $w = e_{-k+1}|\cdots |e_{k-1}$, we have: 
\[
e_a = \begin{cases} d_{a+1} &\mbox{if } a \geq j, \\
				   d_{a-1} &\mbox{if } a \leq -j.				 
				   \end{cases}
\]
Further, all but the newly formed blocks between $d_{-j-1}$ and $d_{j+1}$ are blocks of both $u$ and $w$.

Now suppose $w=u_1$ is the cover of $u$ obtained by merging the pair in the smallest position. There are two cases to consider: either $i \geq 0$ or $i < 0$. 

Suppose first $i \geq 0$, so that the merge from $u$ to $u_1$ is positive. If the smallest merge of $u_1$ is also positive, it must be in a position $a \geq j$ as with types $A_{n-1}$ and $B_n$. If smallest merge of $u_1$ (with respect to $v$) is negative, its position is at least $1+|u_1|$. Thus we have $\lambda_D(u,u_1) = j \leq |u| = |u_1|+1 \leq \lambda_D(u_1,u_2)$. 

Now suppose that $i < 0$, so that the merge is negative. Then the position of $(d_i, d_j)$ is $j + |u|$ and all merging pairs of $u_1$ must also be negative (else, there was a positive merge possible in $u$, contradicting minimality). Again, because $j+|u|$ was minimal in $u$, the smallest position of a merging pair of $u_1$ must then be in a position $a + |u|$, with $a \geq j$. Hence, $\lambda_D(u,u_1) \leq \lambda_D(u_1,u_2)$ in this case as well.

The same logic applies to every cover $u_{i-1} \prec  u_i$, so this shows that the chain $c$ is a rising chain:
\[ \lambda_D(u,u_1) \leq \lambda_D(u_1,u_2) \leq \cdots \leq \lambda_D(u_{r-1},v).\]

To see that any other maximal chain in $[u,v]$ has a fall, we argue once more by induction on the length of the interval. Suppose $w\neq u_1$ is a cover of $u$. By induction, we suppose the lexicographically first chain, $c': w=w_0 \prec  w_1 \prec  \cdots \prec  w_{r-1} = v$, is the only rising chain in $[w,v]$. As before, it will suffice to show that $\lambda_D(u,w) > \lambda_D(w,w_1)$.

There are two cases to consider, depending on whether the merge from $u$ to $w$ is positive or negative. If the merge is positive, then, as was the case with type $B_n$, the smallest merging pair of $w$ is in a position $b \leq j-1$, so $\lambda_D(u,w) = j > b = \lambda_D(w,w_1)$.

If the merge from $u$ to $w$ is negative, then $\lambda_D(u,w) = j + |u|$. If the merge from $w$ to $w_1$ is positive, its position is at most $|w| = |u|-1 < j+|u|$ and we are done. If the merge from $w$ to $w_1$ is negative as well, its position is $b+|w|$ with $b \leq j-1$. Hence $\lambda_D(w,w_1) = b+|w| < j + |w| = j+ |u|-1 < \lambda_D(u,w)$.

This shows $\lambda_D(u,w) > \lambda_D(w,w_1)$ in all cases, and thus every saturated chain $u \prec w \prec \cdots \prec v$ must have a fall.

This proves Proposition \ref{prp:ELD}, and together with Propositions \ref{prp:ELA} and \ref{prp:ELB}, establishes Theorem \ref{thm:EL}.

We now transition from EL-labelings to a different property for graded posets.

\section{Symmetric boolean decomposition}\label{sec:sbd}

Suppose $(P,\leq)$ is a graded poset of rank $n$, and let $\B_j$ (in Roman font to distinguish it from the hyperoctahedral group $B_j$) denote the boolean algebra on $j$ elements. 

\begin{defn}[Symmetric boolean decomposition]\label{def:sbd}
Suppose $\{ P_1,\ldots,P_k\}$ is a partition of $(P,\leq)$, i.e., $P_i\cap P_j = \emptyset$, $\cup_i P_i = P$. Moreover, for each $i = 1,\ldots,k$ suppose there is a number $j$, $0\leq j\leq n/2$, and a bijective map $\rho_i : \B_{n-2j} \to P_i$ that takes cover relations to cover relations and sends elements of rank $r$ in $\B_{n-2j}$ to elements of rank $j+r$ in $P$.

Then we say the collection of images of the maps $\rho_1,\ldots, \rho_k$ provides a \emph{symmetric boolean decomposition} of $(P,\leq)$.
\end{defn}

That is, each induced poset $(P_i,\leq)$ is a saturated subposet in $(P,\leq)$ (i.e., all its cover relations are covers in $P$) with $2^{n-2j}$ elements (for some $j$) and containing a copy of $\B_{n-2j}$, plus possibly some more relations.

For example, in Figure \ref{fig:SBD}, posets (a) and (b) have a symmetric boolean decomposition (in bold), while (c) and (d) do not. Note that $\B_3$ is a proper subposet of the corresponding induced subposet in (a), and that that poset (c) differs from (b) only in one cover relation.

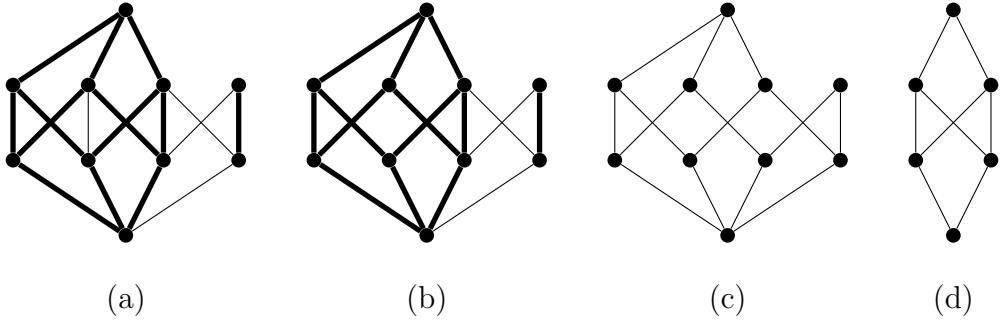
\begin{figure}
\[
\begin{tikzpicture}
\draw (-4,0) node
 {\begin{tikzpicture}[scale=.5]
 \tikzstyle{tight}=[circle,fill=black,scale=.5]
 \draw (0,0) node[tight] (0) {};
 \draw (-3,2) node[tight] (a1) {};
 \draw (-1,2) node[tight] (a2) {};
 \draw (1,2) node[tight] (a3) {};
 \draw (3,2) node[tight] (a4) {};
 \draw (-3,4) node[tight] (b1) {};
 \draw (-1,4) node[tight] (b2) {};
 \draw (1,4) node[tight] (b3) {};
 \draw (3,4) node[tight] (b4) {};
 \draw (0,6) node[tight] (1) {};
 \draw [line width=2] (0)--(a1);
 \draw [line width=2] (0)--(a2);
 \draw [line width=2] (0)--(a3);
 \draw (0)--(a4);
 \draw [line width=2] (1)--(b1);
 \draw [line width=2] (1)--(b2);
 \draw [line width=2] (1)--(b3);
 \draw [line width=2] (b1)--(a1);
 \draw [line width=2] (b2)--(a1);
 \draw (b2)--(a2);
 \draw [line width=2] (b1)--(a2);
 \draw [line width=2] (b3)--(a2);
 \draw [line width=2] (b2)--(a3);
 \draw [line width=2] (b3)--(a3);
 \draw (b4)--(a3);
 \draw (b3)--(a4);
 \draw [line width=2] (b4)--(a4);
 \end{tikzpicture}};
\draw (-4,-2) node[below] {(a)};
\draw (0,0) node
 {\begin{tikzpicture}[scale=.5]
 \tikzstyle{tight}=[circle,fill=black,scale=.5]
 \draw (0,0) node[tight] (0) {};
 \draw (-3,2) node[tight] (a1) {};
 \draw (-1,2) node[tight] (a2) {};
 \draw (1,2) node[tight] (a3) {};
 \draw (3,2) node[tight] (a4) {};
 \draw (-3,4) node[tight] (b1) {};
 \draw (-1,4) node[tight] (b2) {};
 \draw (1,4) node[tight] (b3) {};
 \draw (3,4) node[tight] (b4) {};
 \draw (0,6) node[tight] (1) {};
 \draw [line width=2] (0)--(a1);
 \draw [line width=2] (0)--(a2);
 \draw [line width=2] (0)--(a3);
 \draw (0)--(a4);
 \draw [line width=2] (1)--(b1);
 \draw [line width=2] (1)--(b2);
 \draw [line width=2] (1)--(b3);
 \draw [line width=2] (b1)--(a1);
 \draw [line width=2] (b2)--(a1);
 \draw [line width=2] (b1)--(a2);
 \draw [line width=2] (b3)--(a2);
 \draw [line width=2] (b2)--(a3);
 \draw [line width=2] (b3)--(a3);
 \draw (b4)--(a3);
 \draw (b3)--(a4);
 \draw [line width=2] (b4)--(a4);
 \end{tikzpicture}};
\draw (0,-2) node[below] {(b)};
\draw (4,0) node
 {\begin{tikzpicture}[scale=.5]
 \tikzstyle{tight}=[circle,fill=black,scale=.5]
 \draw (0,0) node[tight] (0) {};
 \draw (-3,2) node[tight] (a1) {};
 \draw (-1,2) node[tight] (a2) {};
 \draw (1,2) node[tight] (a3) {};
 \draw (3,2) node[tight] (a4) {};
 \draw (-3,4) node[tight] (b1) {};
 \draw (-1,4) node[tight] (b2) {};
 \draw (1,4) node[tight] (b3) {};
 \draw (3,4) node[tight] (b4) {};
 \draw (0,6) node[tight] (1) {};
 \draw (0)--(a1);
 \draw (0)--(a2);
 \draw (0)--(a3);
 \draw (0)--(a4);
 \draw (1)--(b1);
 \draw (1)--(b2);
 \draw (1)--(b3);
 \draw (b1)--(a1);
 \draw (b2)--(a1);
 \draw (b1)--(a2);
 \draw (b3)--(a2);
 \draw (b2)--(a3);
 \draw (b4)--(a3);
 \draw (b3)--(a4);
 \draw (b4)--(a4);
 \end{tikzpicture}};
\draw (4,-2) node[below] {(c)};
\draw (7,0) node
 {\begin{tikzpicture}[scale=.5]
 \tikzstyle{tight}=[circle,fill=black,scale=.5]
 \draw (0,0) node[tight] (0) {};
 \draw (-1,2) node[tight] (a1) {};
 \draw (1,2) node[tight] (a2) {};
 \draw (-1,4) node[tight] (b1) {};
 \draw (1,4) node[tight] (b2) {};
 \draw (0,6) node[tight] (1) {};
 \draw (0)--(a1);
 \draw (0)--(a2);
 \draw (1)--(b1);
 \draw (1)--(b2);
 \draw (b1)--(a1);
 \draw (b2)--(a1);
 \draw (b1)--(a2);
 \draw (b2)--(a2);
 \end{tikzpicture}};
\draw (7,-2) node[below] {(d)};
\end{tikzpicture}
\]
\caption{Posets with and without symmetric boolean decompositions.}
\label{fig:SBD}
\end{figure}

It is well known that a boolean algebra admits a symmetric chain decomposition (see, e.g., \cite{GK}), hence the following observation.

\begin{obs}\label{obs:scd}
If $(P,\leq)$ admits a symmetric boolean decomposition, then it admits a symmetric chain decomposition.
\end{obs}

Beyond having a symmetric chain decomposition, and the obvious rank symmetry, there must be ``enough" vertices of each rank for a given poset to admit a symmetric boolean decomposition. (All the posets in Figure \ref{fig:SBD} have symmetric chain decompositions.) To begin, if a rank $n$ poset has a unique maximum and minimum, there must be at least $\binom{n}{k}$ elements of rank $k$. But if there are more than $n$ elements of rank 1, there must be at least $\binom{n}{k} + \binom{n-2}{k-1}$ elements of rank $k\geq 2$, and so on. The notion of \emph{$\gamma$-nonnegativity} encapsulates this necessary numeric condition.

\begin{obs}\label{obs:gam}
Let $(P,\leq)$ be a graded poset of rank $n$. If $(P,\leq)$ admits a symmetric boolean decomposition, then there exist nonnegative integers $\gamma_j$ such that \[ \sum_{p \in P} t^{\rk(p)} = \sum_{0\leq j\leq n/2} \gamma_j t^j (1+t)^{n-2j}.\] In fact, $\gamma_j$ is the number of subsets $P_i$ in the symmetric boolean decomposition for which $|P_i| = 2^{n-2j}$.
\end{obs}

Note that the set $\{ t^j (1+t)^{n-2j}\}_{0\leq j \leq n/2}$ is a basis for polynomials of degree $n$ with symmetric coefficients. Hence for any such polynomial the coefficients $\gamma_j$ as above are uniquely defined. 

Also note that while having a symmetric chain decomposition implies that the rank generating function for $(P,\leq)$ is symmetric and unimodal, $\gamma$-nonnegativity is a stronger condition. This property has been of some interest since Gal conjectured that the $h$-polynomials of flag spheres are $\gamma$-nonnegative \cite{Gal}. See also \cite{Br, NP, NPT, PRW, Stem}.

In the case of the lattice of noncrossing partitions $NC(n) \cong NC(A_{n-1})$, its rank generating function is known to be the $h$-polynomial of the \emph{associahedron}. As Postnikov, Reiner, and Williams observe, $\gamma$-nonnegativity of its $h$-polynomial follows from Simion and Ullman's symmetric boolean decomposition of $NC(n)$. See \cite[Proposition 11.14]{PRW}, Sections 2 and 3 of \cite{SU}, and \cite{BP}. The rank generating function of $NC(n)$ has the following combinatorial interpretation, as follows from Observation \ref{obs:rankA} and Proposition \ref{prp:sub}: \[ \sum_{\pi \in NC(n)} t^{\rk(\pi)} = \sum_{w \in S_n(231)} t^{d(w)}.\] 

Similarly, Observation \ref{obs:rankA} implies the rank generating function of the shard intersection order of the symmetric group $S_n$ is the \emph{Eulerian polynomial} \[A_{n-1}(t)=\sum_{w\in S_n} t^{d(w)},\] which is known to be the $h$-polynomial of the \emph{permutahedron}. Foata and Sch\"utzenberger were the first to show that the Eulerian polynomials are $\gamma$-nonnegative \cite[Th\'eor\`eme 5.6]{FS}. Their result was given a beautiful combinatorial proof by Foata and Strehl in \cite{FSt}. The key idea is an action we call ``valley-hopping" that we will describe in Section \ref{sec:proof}. Valley-hopping and analogous actions have been used in several papers, e.g., \cite{Br, PRW, SWG}. We will use valley-hopping to prove Theorem \ref{thm:SBD}, that the shard intersection order $(S_n,\leq)$ admits a symmetric boolean decomposition which restricts to a symmetric boolean decomposition for $(S_n(231),\leq) \cong (NC(n),\leq)$.

\section{Valley-hopping and type $A_{n-1}$ shard intersections}\label{sec:proof}

We now return to the case of the symmetric group and the proof of Theorem \ref{thm:SBD}.

A wonderful combinatorial proof of the $\gamma$-nonnegativity of the Eulerian polynomial is given by Foata and Strehl's action of ``valley-hopping" as illustrated in Figure \ref{fig:hop}. (There is a similar notion due to Shapiro, Woan, and Getu \cite{SWG}. See also \cite[Section 11]{PRW} and \cite{Br}.) Here we draw a permutation as a ``mountain range", so that peaks and valleys form the upper and lower limits of the decreasing runs. By convention, we have points at infinity on the far left and far right.

\begin{figure}
\[
\begin{tikzpicture}[>=stealth',bend angle=45,auto,xscale=.4,yscale=.5]
\tikzstyle{state}=[circle,draw=black,minimum size=4mm]
\draw (1,11) node[state] (1) {$\infty$};
\draw (4,8) node[state] (2) {$8$};
\draw (6,6) node[state] (3) {$6$};
\draw (10,2) node[state] (4) {$2$};
\draw (15,7) node[state] (5) {$7$};
\draw (18,4) node[state] (6) {$4$};
\draw (21,1) node[state] (7) {$1$};
\draw (23,3) node[state] (8) {$3$};
\draw (25,5) node[state] (9) {$5$};
\draw (29,9) node[state] (10) {$9$};
\draw (31,11) node[state] (11) {$\infty$};
\draw[line width=4] (1)--(2);
\draw[line width=4] (2)--(3);
\draw[line width=4] (3)--(4);
\draw[line width=4] (5)--(6);
\draw[line width=4] (6)--(7);
\draw[dotted] (4)--(5);
\draw[dotted] (7)--(8);
\draw[dotted] (8)--(9);
\draw[dotted] (9)--(10);
\draw[dotted] (10)--(11);
\draw[dashed, <->] (3,9) to [out=15, in=165] (10);
\draw[dashed, <->] (2) to [out=10, in=170] (28,8);
\draw[dashed, <->] (3) to [out=15, in=165] (14,6);
\draw[dashed, <->] (17,5) to [out=15, in=165] (9);
\draw[dashed, <->] (6) to [out=15, in=165] (24,4);
\draw[dashed, <->] (19,3) to [out=10, in=170] (8);
\end{tikzpicture}
\]
\caption{The mountain range view of the permutation $w=862741359$.}\label{fig:hop}
\end{figure}
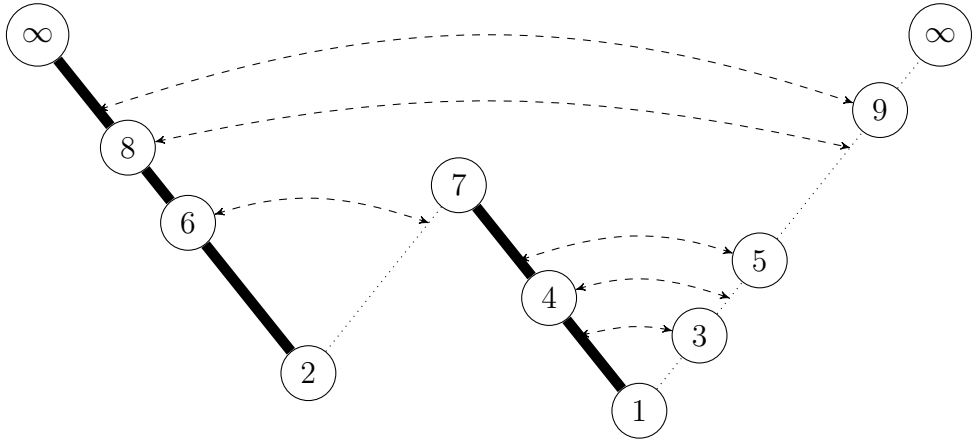

Formally, given $w = w(1)\cdots w(n) \in S_n$, we say a letter $w(i)$ is a \emph{peak} if $w(i-1) < w(i) > w(i+1)$ and it is a \emph{valley} if $w(i-1)> w(i) < w(i+1)$. Otherwise we say $w(i)$ is \emph{free}. Using the convention that $w(0) = w(n+1) = \infty$, we see that $w$ cannot begin or end with a peak.

We partition $S_n$ into equivalence classes according to the following action on free letters. Denote the set of free letters of $w$ as \[F(w) = \{ 1\leq w(i) = j\leq n : w(i-1) < w(i) < w(i+1) \mbox{ or } w(i-1) > w(i) > w(i+1) \}.\] If $w(i) = j$ is free, then $H_j(w)$ denotes the permutation obtained by moving $j$ directly across the adjacent valley(s) to the nearest mountain slope of the same height. More precisely, we have the following.

\begin{defn}[Valley hopping]
Let $w \in S_n$, and let $w(i)=j$ be a free letter of $w$. Define the operator $H_j(w)$ as follows:
\begin{itemize}
\item if $w(i)=j$ lies on a downslope, $w(i-1) > w(i) > w(i+1)$, we find the smallest $k>i$ such that $w(k) < j < w(k+1)$, and \[ H_j(w) = w(1)\cdots w(i-1)w(i+1)\cdots w(k)\, j\, w(k+1) \cdots w(n),\]
\item if $w(i)=j$ lies on an upslope, $w(i-1) < w(i) < w(i+1)$, we find the largest $k<i$ such that $w(k-1) > j > w(k)$, and \[ H_j(w) = w(1)\cdots w(k-1)\, j\, w(k) \cdots w(i-1)w(i+1)\cdots w(n).\]
\end{itemize}
\end{defn}

Clearly, if $j, l\in F(w)$,  $H_j^2(w) = w$ and  $H_j (H_l(w)) = H_l(H_j(w))$. Thus, for any $J = \{j_1,\ldots,j_k\} \subseteq F(w)$, we can define the operation $H_J(w) = H_{j_1}\cdots H_{j_k}(w)$. Also, observe that $F(H_J(w)) = F(w)$, i.e., $H_J(w)$ has the same set of free letters as $w$.

Define a relation $v \sim w$ if there is a sequence of hops on free letters that transforms $w$ into $v$. It is easy to see that $\sim$ is an equivalence relation. Let $P_w$ denote the hop-equivalence class of $w$. If $w$ has $r$ peaks, it has $r+1$ valleys, and hence $n-1-2r$ free letters. Therefore we see that $|P_w| = 2^{n-1-2r}$. 

For each such class, there is a unique element with the minimal number of descents, $r$, corresponding to having each free letter lie on an upslope. Let $\widehat{S_n}$ denote the set of these descent-minimal representatives: \[ \widehat{S_n} = \{ w \in S_n : w(1) < w(2) \mbox{ and if } w(i-1) > w(i), \mbox{ then } w(i) < w(i+1) \}.\] These are the permutations for which every descent is also a peak. 

It is easy to see (and this is the key to Foata and Strehl's proof of $\gamma$-nonnegativity of the Eulerian polynomial) that if $j$ is on an upslope it is \emph{not} in a descent position, while if it is on a downslope it \emph{is} in a descent position. Hence the following observation.

\begin{obs}\label{obs:rankJ}
Let $w \in \widehat{S_n}$ have $r$ descents and $J\subseteq F(w)$. Then $H_J(w)$ has $r+|J|$ descents. In particular, if $w$ has rank $r$ in $(S_n,\leq)$ then $H_J(w)$ has rank $r+|J|$ in $(S_n,\leq)$.
\end{obs}

We claim that the collection of all hop-equivalence classes of elements in $\widehat{S_n}$, \[ \mathcal{P} = \{ P_w : w \in \widehat{S_n}\},\] gives a symmetric boolean decomposition of $(S_n,\leq)$.

\begin{lem}[Partitions $S_n$]\label{lem:partition}
The collection $\mathcal{P}$ is a partition of $S_n$.
\end{lem}

\begin{proof}
Since valley-hopping defines an equivalence relation, we know the set of all hop-equivalence classes partitions $S_n$. All we need is to show that for every $w \in S_n$, there is a unique $u \in \widehat{S_n}$ such that $w\sim u$. The desired element is given as follows. Let $J\subseteq F(w)$ be the set of all free letters that lie on a downslope. Then $u =H_J(w) \in \widehat{S_n}$, and for any $K\neq J$, $H_K(w) \notin \widehat{S_n}$ since it has some letter on a downslope.
\end{proof}

Together with Lemma \ref{lem:partition} and Observation \ref{obs:rankJ}, the following lemma establishes the first claim of Theorem \ref{thm:SBD}.

\begin{lem}[Boolean isomorphism]\label{lem:boole}
Let $w \in \widehat{S_n}$ with $r$ descents. Then the induced poset $(P_w, \leq)$ is isomorphic to $\B_{n-1-2r}$. 
\end{lem}

\begin{proof}
To prove the lemma we will first prove that if $J \subseteq K \subseteq F(w)$, then $H_J(w) \leq H_K(w)$ in the shard intersection order $(S_n,\leq)$. It suffices to assume $u \in P_w$ with $j$ on an upslope and show $u < H_j(u)$ in the shard intersection order.

If $u(i) = j$ is on an upslope in $u$, then it forms its own (singleton) decreasing run, while in $H_j(u)$, the letter $j$ is part of a block of size at least two (including at least the nearest valley to its right). Thus we see that $u$ refines $H_j(u)$. Since $j$ is the only element to move and its comparability to other elements in the pre-order is preserved (it can't move left of a block it was previously to the right of only by jumping over lower lying valleys), we see the relations on the decreasing runs of $u$ and $H_j(u)$ are consistent. So from Proposition \ref{prp:order} we see $u < H_j(u)$ in $(S_n, \leq)$.

Now suppose $J \nsubseteq K$. We will show $u=H_J(w) \nleq H_K(w)=v$. In particular let $j \in J$, $j\notin K$. Then in $u$, $j$ lies on a downslope, while in $v$, $j$ lies on an upslope. This means in particular that $j$ forms a singleton block in $v$, while it is part of a block of size at least two in $u$. Thus, as pre-orders, $u$ cannot be a refinement of $v$, and Proposition \ref{prp:order} shows $u \nleq v$ in $(S_n,\leq)$.
\end{proof}

We now show the second claim of Theorem \ref{thm:SBD}, that our symmetric boolean decomposition restricts to noncrossing partitions. Because of Lemma \ref{lem:boole} and Observation \ref{obs:rankJ}, we need only show that 
\[ \mathcal{P}(231) = \{ P_w : w \in \widehat{S_n} \cap S_n(231) \}\] is a partition of $S_n(231)$, the set of 231-avoiding permutations. We already know from Lemma \ref{lem:partition} that the classes $P_w$, $w \in \widehat{S_n}$, are pairwise disjoint, so it only remains to check that if $w \in S_n(231)$, then its hop-equivalence class $P_w$ is entirely contained in $S_n(231)$. This is established with the following lemma.

\begin{lem}[Restriction to 231-avoiders]
Let $w \in S_n(231)$. Then $P_w \subseteq S_n(231)$. That is, valley hopping gives an equivalence relation on the set of $231$-avoiding permutations.
\end{lem}

\begin{proof}
It suffices to show that if $w \notin S_n(231)$, $H_s(w) \notin S_n(231)$ for all $s \in F(w)$. 

Suppose $w$ contains the pattern $231$, i.e., there is a triple $i < j < k$ with $w(k) < w(i) < w(j)$. Without loss of generality, we may assume that $w(j)$ is a peak. (If not, there is necessarily a peak $w(j')$ with $i<j'<j$ and $w(j')>w(j)$.) If neither $w(i)$ nor $w(k)$ are free, or if $s \notin \{w(i),w(k)\}$, then we are done, as $H_s(w)$ leaves the relative positions of $w(i), w(j),$ and $w(k)$ unchanged. 

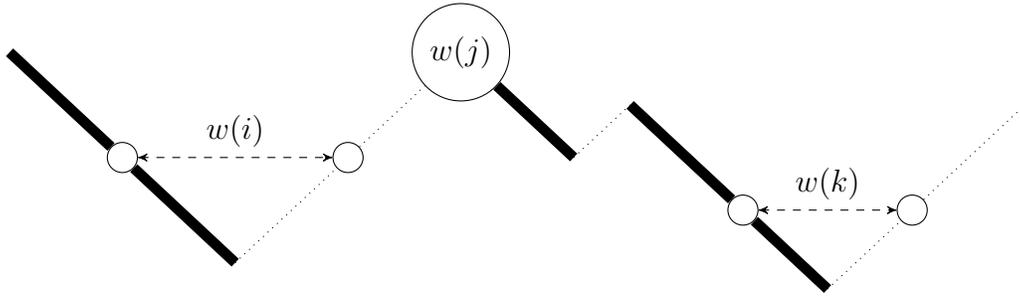
\begin{figure}
\[
\begin{tikzpicture}[>=stealth',bend angle=45,auto,xscale=.75,yscale=.7]
\tikzstyle{state}=[circle,draw=black,minimum size=4mm]
\draw (0,0) node[state] (1) {};
\draw (4,0) node[state] (2) {};
\draw (6,2) node[state] (3) {$w(j)$};
\draw (11,-1) node[state] (4) {};
\draw (14,-1) node[state] (5) {};
\draw[line width=4] (-2,2)--(1);
\draw[line width=4] (1)--(2,-2);
\draw[dotted] (2,-2)--(2);
\draw[dotted] (2)--(3);
\draw[line width=4] (3)--(8,0);
\draw[dotted] (8,0)--(9,1);
\draw[line width=4] (9,1)--(4);
\draw[line width=4] (4)--(12.5,-2.5);
\draw[dotted] (12.5,-2.5)--(5);
\draw[dotted] (5)--(16,1);
\draw[dashed, <->] (1) -- node[auto=left] {$w(i)$} (2);
\draw[dashed, <->] (4) -- node[auto=left] {$w(k)$} (5);
\end{tikzpicture}
\]
\caption{Valley hopping preserves the pattern 231.}\label{fig:hopproof}
\end{figure}

Now suppose $s = w(i)$. Valley hopping would never allow $w(i)$ to move to the right of $w(j)$ (since $i < j$ and $w(i) < w(j)$), so $H_s(w)\notin S_n(231)$. Similarly, if $s=w(k)$, $H_s(w) \notin S_n(231)$, since $w(k)$ would not be able to move to the left of $w(j)$ (since $k > j$ and $w(k) < w(j)$). See Figure \ref{fig:hopproof}.
\end{proof}

We remark that the symmetric boolean decomposition of $(NC(n),\leq)$ given here, as the image of $(S_n(231),\leq)$ under the map $\phi$ from Section \ref{sec:noncrossing}, is distinct from Simion and Ullman's decomposition. For example, in $NC(4)$, our decomposition describes the maximal boolean interval as: \[ P_{1234} = \{ 1|2|3|4, 21|3|4, 31|2|4, 41|2|3, 321|4, 421|3, 431|2, 4321\},\] whereas Simion and Ullman's decomposition gives the maximal boolean interval as: \[ \{ 1|2|3|4, 21|3|4, 1|32|4, 1|2|43, 321|4, 21|43, 1|432, 4321\}.\]

\section{Further questions}\label{sec:generalization}

Apart from checking shellability for Coxeter groups not of type $A_{n-1}, B_n$, or $D_n$, the natural next task is to exhibit a symmetric boolean decomposition for the shard intersection order $(W,\leq)$ of any finite Coxeter group $W$. Ideally, this would restrict to a symmetric boolean decomposition of $(NC(W),\leq)$. 

There is some hope for such a decomposition, as it is known that the $W$-Eulerian polynomial is $\gamma$-nonnegative for any $W$ \cite[Appendix]{Stem}. One way to proceed in finding a symmetric boolean decomposition of the shard intersection order $(W,\leq)$ is to take a purely combinatorial approach. The models provided in Section \ref{sec:poset} give a first step in this direction.  With luck one would lift a combinatorial proof of $\gamma$-nonnegativity for the $W$-Eulerian polynomial as we have done here for $W=S_n$. 

For instance, Stembridge's proof of $\gamma$-nonnegativity in type $B_n$ is quite nice. The idea is to show that the distribution of descents among all signed permutations with the same underlying permutation is boolean. Unfortunately this partition of $B_n$ is not compatible with the shard intersection order. For instance, following Stembridge's argument would give $1432$ and $14\bar 3 2$ in the same equivalence class, but $C(1432) = \{ x_4 = x_3 = x_2\}$ while $C(14\bar 3 2) = \{ -x_2 \leq x_3 = -x_4 \leq \pm x_1, 0 \leq x_4 = -x_3 \leq x_2\}$. Clearly neither cone contains the other.

Of course a uniform proof that $(W,\leq)$ has a symmetric boolean decomposition would be best. We remark that at present there is no uniform proof of $\gamma$-nonnegativity for $W$-Eulerian polynomials.

Another reason to believe $(W,\leq)$ might have a symmetric boolean decomposition is that, just as Simion and Ullman show for the classical (type $A_{n-1}$) lattice of noncrossing partitions, the lattice of noncrossing partitions of type $B_n$ has a symmetric boolean decomposition. This fact is proven by Hersh \cite[Theorem 8]{H}, using the notion of an $R^*S$-labeling---a strong sort of chain labeling together with a certain $S_n$-action on maximal chains. (Our EL-labelings in Section \ref{sec:shell} are not $R^*S$-labelings.) If one can show that the shard intersection order $(W,\leq)$ has an $R^*S$ labeling, then \cite[Theorem 5]{H} implies it has a symmetric boolean decomposition.

Another approach to Hersh's result for $(NC(B_n),\leq)$ proceeds inductively as follows. In \cite[Proposition 12]{Rei}, Reiner shows (as a prelude to proving the existence of a symmetric chain decomposition) that the type $B_n$ noncrossing partition lattices are unions of products of smaller copies of noncrossing partition lattices with the appropriate centers of symmetry.\footnote{Both Hersh and Reiner study other families of lattices of ``signed" noncrossing partitions including a family they refer to as that of ``type $D_n$". While these lattices admit symmetric boolean decompositions, this model has since come to be viewed as the wrong generalization from the point of view of root systems. In particular, it is different from the lattice $(NC(D_n),\leq)$ that lies inside the shard intersection order for $D_n$. See Athanasiadis and Reiner \cite[Remark 4]{AR}, where it is remarked that $NC(D_n)$ admits a symmetric chain decomposition.} Thus, after checking small cases, it suffices to show that the product of two posets with symmetric boolean decompositions admits a symmetric boolean decomposition. This is true, as the following lemma shows. Recall the product of two posets $(P,\leq_P)$ and $(Q,\leq_Q)$ is the poset $(P\times Q, \leq_{P\times Q})$ with partial order $(p,q) \leq_{P\times Q} (p',q')$ if and only if $p\leq_P p'$ and $q\leq_Q q'$.

\begin{lem}\label{lem:prod}
Suppose $(P,\leq_P)$ and $(Q,\leq_Q)$ are posets with symmetric boolean decompositions. Then $(P\times Q, \leq_{P\times Q})$ has a symmetric boolean decomposition.
\end{lem}

\begin{proof}
Suppose $P$ has rank $m$ and $Q$ has rank $n$. Denote their respective decompositions by
\[ \mathcal{P} = \{ P_{i,j} : |P_{i,j}| = 2^{m-2j} \} \mbox{ and } \mathcal{Q}= \{ Q_{k,l} : |Q_{k,l}| = 2^{n-2l} \}.\]
We wish to show $P\times Q$, which has rank $m+n$, has a symmetric boolean decomposition. We claim that the following is such a partition: \[ \mathcal{PQ} = \{ P_{i,j}\times Q_{k,l} \}. \] Clearly $\mathcal{PQ}$ is a partition of $P\times Q$.

Now, as a first check that Definition \ref{def:sbd} is satisfied, we  see \[|P_{i,j}\times Q_{k,l}| = |P_{i,j}|\times |Q_{k,l}| = 2^{m-2j}2^{n-2l} = 2^{m+n-2(j+l)}.\] Say $p$ is the rank-minimal element of $P_{i,j}$, with $\rk(p) = j$ in $(P,\leq_P)$, and $q$ is the rank-minimal element $Q_{k,l}$ in $(Q,\leq_Q)$ with $\rk(q) = l$. Then $(p,q)$ is the unique rank-minimal element of $P_{i,j}\times Q_{k,l}$, with $\rk(p,q) = \rk(p)+\rk(l) = j+l$. Generally, rank is additive in the product, so since $\B_j$ is a subposet of $P_{i,j}$ and $\B_l$ is a subposet of $Q_{k,l}$, we have $\B_{j+l} \cong \B_j \times \B_l$ is a subposet of $P_{i,j}\times Q_{k,l}$, as desired. 
\end{proof}

\begin{cor}
\emph{(\cite{SU} and \cite[Theorem 8]{H}).} 
The noncrossing partition lattices $(NC(n),\leq)$ and $(NC(B_n),\leq)$ admit symmetric boolean decompositions.
\end{cor}

Recall that, for any finite Coxeter group $W$, the rank generating function of the $W$-noncrossing partition lattice is the $h$-polynomial of the corresponding $W$-associahedron \cite[Theorem 5.9]{FR}. Thus, by Observation \ref{obs:gam}, we get the following known result.

\begin{cor}
Let $W$ be a Coxeter group of type $A_{n-1}$ or $B_n$. Then the $h$-polynomial of the $W$-associahedron is $\gamma$-nonnegative.
\end{cor}

We remark that the coefficients of the $h$-polynomials of $W$-associahedra have nice formulas (these are $W$-Narayana numbers), and $\gamma$-nonnegativity for the classical types (including $D_n$) can be verified directly from these formulas. 
From \cite[Figure 5.12]{FR} it is also straightforward to check that $\gamma$-nonnegativity holds for the $W$-associahedra of exceptional type.

We finish by observing that while we have provided two necessary conditions for a poset to have a symmetric boolean decomposition (Observation \ref{obs:scd} and Observation \ref{obs:gam}), we have only one sufficient condition (Lemma \ref{lem:prod}). Hersh's work gives another sufficient condition: the existence an $R^*S$ labeling for the poset \cite[Theorem 5]{H}. Just as others have given sufficient conditions for symmetric chain decompositions to exist \cite{GK,Gr}, so it may be of general interest to find other sufficient conditions for the existence symmetric boolean decompositions.

\end{document}